\documentclass{article}

\usepackage[english]{babel}

\usepackage[letterpaper,top=2cm,bottom=2cm,left=3cm,right=3cm,marginparwidth=1.75cm]{geometry}

\usepackage{amsmath}
\usepackage{amssymb}
\usepackage{amsthm}
\usepackage{graphicx}
\usepackage{enumitem}
\usepackage[colorlinks=true, allcolors=blue]{hyperref}
\usepackage[all]{xy}

\usepackage{xcolor}
\usepackage[all]{xy}

\newtheorem{Thm}{Theorem}
\newtheorem{Prop}{Proposition}[section]
\newtheorem{Lem}[Prop]{Lemma}
\newtheorem{Cor}{Corollary}[Prop]
\newtheorem{Corth}[Prop]{Corollary}
\theoremstyle{definition}

\newtheorem{ex}[Prop]{Example}
\theoremstyle{remark}
\newtheorem{rem}[Prop]{Remark}

\newcommand\bmattrix[4]{\left(\begin{array}{cc}#1&#2\\#3&#4\end{array}\right)}
\newcommand\sbmattrix[4]{\textnormal{\scriptsize$\left(\begin{array}{cc}#1&#2\\#3&#4\end{array}\right)$\normalsize}}
\newcommand\sbvecttor[2]{\textnormal{\scriptsize$\left(\begin{array}{cc}#1\\#2\end{array}\right)$\normalsize}}

\title{Two dimensional integral representations via branches of the Bruhat-Tits tree}
\author{
  Aguil\'o-Vidal, Bruno\\
  \texttt{bruno.aguilo@udp.cl}
  \and
    Arenas-Carmona, Luis\\
  \texttt{learenas@u.uchile.cl}
  \and
  Saavedra-Lagos, Mat\'ias\\
  \texttt{matias.saavedra.l@ug.uchile.cl}
}

\begin{document}
\maketitle

\begin{abstract}
We apply the theory of branches in Bruhat-Tits trees, developed in previous 
works by the second author and others, to the study of two 
dimensional representations of finite groups over the ring of integers
of a number field. We provide a general strategy to perform these 
computations, and we give explicit formulas for some particular families.
\end{abstract}

MSC-class: 11R33, 20C10 (Primary)  11R56 (Secondary).

\section{Introduction}

Representation theory is a branch of mathematics that explores how
algebraic structures can be represented as linear transformations of vector
spaces. It has deep connections to various fields, 
including physics, geometry, and number theory.
Since finite group of linear transformations appear naturally in a 
variety of contexts, having a classification of the essentially different
ways a finite group can appear as a group of linear transformation
is an important problem. One that has a natural solution through
the study of group algebras, which, at least over fields of characteristic 
zero, has an extremely simple structure. This allows us to prove easily,
for instance, the finiteness of the number of irreducible representations,
and often compute their number and dimensions. 

Integral representations play an equally significant role in any context 
where discrete structures are considered, like in studying the structure of 
crystals. 
As an example, it is relatively easy to prove that every lattice $\Lambda$
with a symmetry $\sigma$ of order $2$ can be written as the direct sum 
of three lattices, $\Lambda_1$, $\Lambda_{-1}$ and $\Lambda_i$,
where $\sigma$ acts as the identity on $\Lambda_1$, as multiplication by 
$-1$ on $\Lambda_{-1}$, while $\Lambda_i$ is a sum of rank-$2$ lattices
with $\sigma$ acting as the matrix $\sbmattrix 0110$ on each. This gives
a complete classification of the $\mathbb{Z}$-representations of
the group $C_2$ with two elements.

More generally, we are interested in classifying maps 
$\phi:G\rightarrow\mathrm{GL}_n(\mathcal{O}_K)$,
where $\mathcal{O}_K$ is the ring of integers of a number
field $K$, up to a change of basis, 
i.e., up to conjugacy by an invertible matrix in 
$\mathrm{GL}_n(\mathcal{O}_K)$. One possible approach is to imitate
the already successful theory for fields, i.e., study all possible 
modules over the group ring $\mathcal{O}_K[G]$. An example of this 
approach can be found in \cite{Le64}. See also \cite{HR62a},
\cite{HR62b}, \cite{Jo62} and \cite{Ka98}. This is way harder than
the corresponding problem for fields. For instance, when
$R=\mathbb{Z}[C_2]$, the finitely generated
torsion free $R$-modules have the structure
that has already been described above, but only lattices of the form
$\Lambda=\Lambda_i$ are actually projective modules, as there are 
natural exact sequences of the form
$$0\rightarrow \Lambda_\epsilon\rightarrow \Lambda_i
\rightarrow\Lambda_{-\epsilon}\rightarrow0,$$
for $\epsilon\in\{\pm1\}$ that fail to split. This tell us that a
complete understanding of the representation theory for these rings
is significantly more involved, and in fact the theory
of $\mathcal{O}_K[G]$-modules encodes
abundant arithmetic information, as should be apparent in the
sequel.

Since the representation theory over fields is 
already well understood,
there is another path to the study of integral representations.
To classify them first under a number field $K$, and then study
the set of integral representation that are conjugate 
to a given
$K$-representation. This set is potentially empty, see
\cite{cli92} or \cite{Se08} for an example, 
or \cite{Hoff16} for a more
algorithmic approach. This is our path in this work. To study
integral representation in this context, 
we use the correspondence
between maximal orders containing a given representation
and integral representations in a given maximal 
order, as described in 
\S4 below. Then we study the set of maximal 
orders containing one given 
representation or, equivalently, 
containing the  order  generated by its image.
The cases where this order is maximal 
have been studied in \cite{Man18}.
This is not often the case. However, 
tools for the study of the set 
of maximal orders containing a given suborder, in a more 
general setting, have already been developed in 
some of our previous
works like \cite{A13},\cite{AS16},\cite{AAC18} or \cite{AB19}. 
They focus mostly on quaternion algebras, like the
four dimensional matrix algebra $\mathbb{M}_2(K)$,
and that is the reason we limit our
study to two-dimensional representation here. However, 
some of our techniques
can be extended to higher dimensions, and we expect to do so in
future works. Here we focus in adapting the techniques previously
developed to study embeddings of quadratic rings into orders
in quaternion algebras to the precise objective
at hand. This is often simpler when studying representations
of finite groups over their field of definition $K_0$, since 
this theory rely mostly on the arithmetic of cyclotomic fields. 
Our main results take a much simpler form over $K_0$ than over
an arbitrary extension. However, the language of Bruhat-Tits
trees is entirely appropriate to state the more general Theorems.
Due to their technicality, however,
we require some preparation before 
stating them. We delay this until next section.

\subparagraph{Basic notations and facts.}

In all of this work, we assume that $K/\mathbb{Q}$ is a 
finite extension, and $\mathcal{O}_K\subseteq K$ denote
the ring of integers in $K$, which is a Dedekind domain. 
In particular, the non-zero 
fractional $\mathcal{O}_K$-ideals in $K$
form a group under multiplication that we denote by
$\mathcal{I}_K$. The subset of principal 
ideals $a\mathcal{O}_K$ is
a subgroup $\mathcal{P}_K$, and the quotient 
$\mathcal{G}_K=\mathcal{I}_K/\mathcal{P}_K$ is 
called the class group
of $K$. For any finite extension $L/K$ 
there is a multiplicative 
norm function
$N_{L/K}:\mathcal{I}_L\rightarrow\mathcal{I}_K$, 
and the relative class
group is defined by
$$\mathcal{G}_{L/K}=N_{L/K}^{-1}(\mathcal{P}_K)/\mathcal{P}_L
\subseteq\mathcal{G}_L.$$
The order of the group $\mathcal{G}_K$ is called the 
class number of $K$
and denoted $h_K$. Analogously, the order 
of the group $\mathcal{G}_{L/K}$
is called the relative class number of the 
extension $L/K$ and denoted 
$h_{L/K}$. We also use the group 
$\mathcal{G}_K(n)$, which is defined
as the set of classes $\bar{I}\in\mathcal{G}_K$ that satisfy
$\bar{I}^n=e_{\mathcal{G}_K}=\overline{(1)}$.
Equivalently, $\bar{I}\in\mathcal{G}_K(n)$ if and only if
$I^n$ is a principal fractional ideal. We also use 
$h_K(n)$ for the order of  $\mathcal{G}_K(n)$.

By a finite place of a field $K$ we mean an equivalence class
of discrete valuations, which is often denoted by
one representative $\nu$. The completions of $K$
and $\mathcal{O}$ are denoted by $K_\nu$
and $\mathcal{O}_\nu$. When $K$ is a number field,
these places are in correspondence with the maximal
ideals in $\mathcal{O}_K$. The maximal ideal corresponding 
to a valuation $\nu$ is denoted $P[\nu]$. Conversely,
we use $\nu_P$ for the valuation corresponding to a maximal 
ideal $P$. The notation for maximal ideals intend to
make easy to remember the rational prime it contains.
For instance, $\mathbf{2}_1$ is a maximal ideal containing
$2$, and $\nu_{\mathbf{2}_1}$ is a dyadic place. 

The greek letters $\rho$, $\phi$, $\varphi$, $\psi$ and 
variations thereof denote representations of a 
finite group $G$ and take values in the group 
$\mathrm{GL}_2(K)$ of invertible two by two 
matrices with coefficients in $K$. They are called integral
if the image is contained in $\mathrm{GL}_2(\mathcal{O}_K)$.
By convention, we identify any such representation $\phi$ 
with the composition $\mathbf{i}\circ\phi$, where
$\mathbf{i}:\mathrm{GL}_2(K)\hookrightarrow\mathrm{GL}_2(L)$
is the natural inclusion,
for any field extension $L/K$, whenever confusion seems 
unlikely. In particular, this is applied to completions.

Maximal orders in the matrix algebra $\mathbb{M}_2(K)$ 
are denoted by variations of the symbol
$\mathfrak{D}$. Usually $\mathfrak{H}$ denote a non-maximal
order, often an order in a two dimensional subalgebra.
We use the term Artin distance to refer to the distance function defined in \cite[\S1]{A13}. For every pair 
$(\mathfrak{D},\mathfrak{D}')$ of maximal orders,
the Artin distance $\mathrm{d}(\mathfrak{D},\mathfrak{D}')$ 
is an element in the Galois group of a suitable extension $\Sigma/K$ called the spinor class field. It is the identity
if and only if $\mathfrak{D}$ is isomorphic to $\mathfrak{D}'$.

\subparagraph{Trees and branches.}

By the Bruhat-Tits tree of $K$ at a local place $\nu$ we 
mean the
tree $\mathfrak{t}_\nu$ whose vertices are in 
correspondence with
the closed balls in the completion $K_\nu$ with two balls 
being neighbors if one is a maximal sub-ball of the other. The 
linear group  $\mathrm{GL}_2(K_\nu)$ acts on this 
tree via Moebius
Transformations. Alternative definitions of the tree,
as much as several ways to explicitly described the action
are recalled in \S\ref{S5B}. For any matrix $\mathtt{T}\in
\mathrm{GL}_2(K_\nu)$ there is a (possibly empty) maximal
invariant subtree 
$\mathfrak{s}_\nu(\mathtt{T})\subseteq\mathfrak{t}_\nu$.
The vertices of this subtree are precisely the 
invariant vertices.
The definition of the invariant subtree 
$\mathfrak{s}_\nu(\mathtt{T}_1,\dots,\mathtt{T}_n)=
\bigcap_{i=1}^n\mathfrak{s}_\nu(\mathtt{T}_i)$ is analogous. 
The subindex $\nu$ is omitted if it is clear from the context. 
These invariant subtrees have a very simple structure 
when the matrices $\mathtt{T}_i$ have different eigenvalues. 
They contain
precisely the vertices at a fixed distance $d$ or 
less that a central 
path called the stem. The connected 
components of the graph obtained 
by removing all edges in the stem are called side branches. All side branches look equal.
See Fig. \ref{fn1}.\textbf{A}-\textbf{B} for an example.
For instance, the diagonal matrix $\sbmattrix100h$
has a non-empty invariant subtree if and only if the valuation 
$\nu(h)$ is zero, and in that case we have $d=\nu(h-1)$.
Furthermore, the stem, for that matrix, is the path containing
precisely the vertices corresponding to balls centered at $0$.
See Lemma \ref{l61} for details.

Global matrices (or sets of them) have an invariant subtree
at every local place. However, these local side branches 
are trivial, i.e.,
consist of a single vertex, at all but finitely many places.
For a single matrix, the stem can be a vertex, an edge,
or a maximal path, i.e., an infinite path that is infinite
in both directions. Figure \ref{fn1}.\textbf{A} shows an
example of the latter. See \S\ref{S5B} for details.

\begin{figure}
\[
\unitlength 1mm 
\linethickness{0.4pt}
\ifx\plotpoint\undefined\newsavebox{\plotpoint}\fi 
\begin{picture}(85,27)(15,0)
\put(32,5){\textbf{A}}
\put(20.75,12){\line(1,0){28.75}}
\put(16,11.2){$\cdots$}
\put(50,11.2){$\cdots$}
\put(23.75,12){\line(0,1){7.25}}
\multiput(23.75,19)(-.033653846,.069711538){80}
{\line(0,1){.069711538}}
\multiput(23.75,19)(.033653846,.069711538){80}
{\line(0,1){.069711538}}
\put(30.75,12){\line(0,1){7.25}}
\multiput(30.75,19)(-.033653846,.069711538){80}
{\line(0,1){.069711538}}
\multiput(30.75,19)(.033653846,.069711538){80}
{\line(0,1){.069711538}}
\put(37.75,12){\line(0,1){7.25}}
\multiput(37.75,19)(-.033653846,.069711538){80}
{\line(0,1){.069711538}}
\multiput(37.75,19)(.033653846,.069711538){80}
{\line(0,1){.069711538}}
\put(44.75,12){\line(0,1){7.25}}
\multiput(44.75,19)(-.033653846,.069711538){80}
{\line(0,1){.069711538}}
\multiput(44.75,19)(.033653846,.069711538){80}
{\line(0,1){.069711538}}
\put(64.75,12){\line(0,1){7.25}}
\multiput(64.75,19)(-.033653846,.069711538){80}
{\line(0,1){.069711538}}
\multiput(64.75,19)(.033653846,.069711538){80}
{\line(0,1){.069711538}}
\put(74.75,12){\line(0,1){7.25}}
\put(69,5){\textbf{B}}
\put(95,12){\line(0,1){7.25}}
\multiput(95,19)(-.033653846,.069711538){80}
{\line(0,1){.069711538}}
\multiput(95,12)(-.051339286,.033482143){56}
{\line(-1,0){.051339286}}
\multiput(95,19)(-.051339286,.033482143){56}
{\line(-1,0){.051339286}}
\multiput(95,19)(.033653846,.069711538){80}
{\line(0,1){.069711538}}
\put(92,13.8){\line(0,1){7.25}}
\multiput(92,20.8)(-.033653846,.069711538){80}
{\line(0,1){.069711538}}
\multiput(92.2,24.7)(-.051339286,.033482143){56}
{\line(-1,0){.051339286}}
\multiput(97.9,24.7)(-.051339286,.033482143){56}
{\line(-1,0){.051339286}}
\multiput(93.4,23)(.033653846,.069711538){50}
{\line(0,1){.069711538}}
\put(93,5){\textbf{C}}
\put(94,11.5){$\bullet$}
\put(64,11.5){$\bullet$}
\put(73.8,11.5){$\bullet$}
\end{picture}
\]
\caption{A typical branch (\textbf{A}). A side branches
of the latter subtree and a smaller side branch (\textbf{B}).
The product of the two previous side branches (\textbf{C}).}
\label{fn1}
\end{figure}

We make intensive use of products of Bruhat-Tits trees,
or some of their subtrees, at different places.
For instance, the product $\mathbb{B}$ of the side
branches plays an important role in our main results.
These products are simpler to define if graphs
are considered as topological spaces with distinguished points
called vertices and distinguished subsets called edges.
See Fig. \ref{fn1}.\textbf{C} for an example. 
By vertices of the product, we refer
to those elements having a vertex in every coordinate.
The example in Fig. \ref{fn1}.\textbf{C} has 8 vertices.
Most of the products that we actually consider are finite,
in the sense that all but a finite number of factors
consist of a single vertex, as it is the case for the product
of side branches. The product $\mathbb{B}$ of the side
branches has a unique point $v=(v_\nu)_\nu$ whose every 
coordinate $v_\nu$ is a stem vertex.
This is called the anchor vertex.
Stem vertices in Fig. \ref{fn1}.\textbf{B} and the anchor 
vertex in Fig. \ref{fn1}.\textbf{C} are marked by
bullets ``$\bullet$''.

For any two vertices $v=(v_\nu)_\nu$ and
$w=(w_\nu)_\nu\in\mathbb{B}$, the information 
of the local graph distance $d_\nu$ between the coordinates
$w_\nu$ and $v_\nu$ can be collected in an ideal 
$I_{(v,w)}=\prod_\nu P[\nu]^{d_\nu}$,
called the distance ideal between the two vertices.
Its valuation at $\nu$ is, by definition, $d_\nu$.

\section{Main results}\label{MR}

In what follows, we fix a finite group $G$, and a 
representation $\phi_0:G\rightarrow
\mathrm{GL}_2(\mathcal{O}_K)$. Let $\tilde{\Phi}$ be the set 
of all representations $\phi:G\rightarrow
\mathrm{GL}_2(\mathcal{O}_K)$ that are 
$\mathrm{GL}_2(K)$-conjugate to $\phi_0$. Let $\Phi$ be the set
of $\mathrm{GL}_2(\mathcal{O}_K)$-conjugacy classes in
$\tilde{\Phi}$. Let $\tilde{\Omega}$ be the set of
maximal orders $\mathfrak{D}\subseteq\mathbb{M}_2(K)$ 
containing $\phi_0(G)$ that are isomorphic to
$\mathbb{M}_2(\mathcal{O}_K)$. Let
$\mathfrak{L}=\mathfrak{L}_{\phi_0}$ be the centralizer of the
representation $\phi_0$. Finally, let $\Omega$ be the set 
 of $\mathfrak{L}^*$-conjugacy classes of orders 
in $\tilde{\Omega}$. With these notations the general tool 
that we develop in this work to describe integral 
representations is the following:

\begin{Thm}\label{percc}
    In the preceding notations, assume 
    $\phi_0$ is not trivial.
    There exists a natural surjective map 
    $f:\Phi\rightarrow\Omega$ for which
    the pre-image $f^{-1}[\mathfrak{D}]$
    of the class $[\mathfrak{D}]\in\Omega$ 
    of an order
    $\mathfrak{D}\in\tilde{\Omega}$ 
    is in correspondence 
    with the quotient group 
    \begin{equation}\label{defna}
        \tilde{N}_\mathtt{a}=\frac{N}{
        \mathrm{GL}_2(\mathcal{O}_K)(N\cap 
        \mathtt{a}\mathfrak{L}^*\mathtt{a}^{-1})},
    \end{equation}
    where $N\subseteq\mathrm{GL}_2(K)$ is the normalizer of $\mathbb{M}_2(\mathcal{O}_K)$, and 
    $\mathtt{a}\in\mathrm{GL}_2(K)$ is any matrix satisfying
    $\mathfrak{D}=\mathtt{a}^{-1}\mathfrak{D}_0 \mathtt{a}$.
    In particular, the order of this group is independent on 
    the choice of $\mathtt{a}$.
\end{Thm}

Our general strategy is to describe the sets $\tilde{\Omega}$
and $\Omega$ for different families of representations in terms
of subgraphs of Bruhat-Tits trees, on which the action of
the group $\mathfrak{L}^*$ is often simple to describe.
Maximal orders are described in terms of vertices. We need
to tell apart the vertices that do correspond to maximal orders
isomorphic to $\mathbb{M}_2(\mathcal{O}_K)$, and this
is achieved thanks to the existence of the Artin distance, 
which classifies maximal orders into isomorphism 
classes and it is easy to read from the graph. Finally,
we need to compute the quotients $\tilde{N}_\mathtt{a}$
corresponding to each vertex. Since the quotient 
$N/\mathrm{GL}_2(\mathcal{O}_K)$ is well understood in terms
of the ideal group, the final step of the computation
reduces to the study of certain ideal classes related
to the group $N\cap\mathtt{a}\mathfrak{L}^*\mathtt{a}^{-1}$.
The detail of this last step vary according to the 
representation family at hand. In particular, it depends
on the nature of the centralizer $\mathfrak{L}$, which is
trivial for absolutely irreducible representations, but not
for representations of abelian groups. The same can be said
about the $\mathfrak{L}^*$-action on vertices. For abelian
groups, the $\mathfrak{L}^*$-action on stem vertices is easy 
to describe in terms of ideal class groups, so we are
left with the action on side branches, which can be 
described in terms of congruence relations on groups of units.
See \S\ref{S6} for details.

Our first pecific result concerns representations of the form $\rho=\sbmattrix{\chi_1}00{\chi}$, where $\chi_1$ and $\chi$
are one dimensional representations. Dividing by
$\chi_1$, whose values are roots of unity and therefore
unit in $\mathcal{O}_K$, we can assume that $\chi_1=1$,
so $\rho=\rho_\chi$, where $\rho_\chi=\sbmattrix100{\chi}$.
Since any finite group of root of unity is cyclic, the image
of $\rho$ is generated by a matrix $\sbmattrix100{\eta}$,
where $\eta$ is a root of unity.

 \begin{Thm}\label{t5}
    Assume $K\supseteq\mathbb{Q}[\eta]$.
    Let $\mathbb{B}$ be the product of the local
    side branches corresponding to the matrix 
    $\mathtt{D}=\sbmattrix100\eta$. Assume that
    $\mathcal{O}_K^*$ act on $\mathbb{B}$ 
    by multiplication with $\mathbf{O}$ orbits of vertices.
    Then the number of 
    $\mathrm{GL}_2(\mathcal{O}_K)$-conjugacy 
    classes of $\mathcal{O}_K$-representations of $G$ that are
    conjugate to $\rho_\chi$ over $K$ is $\mathbf{O}h_K$. 
\end{Thm}

\begin{ex}
    If the order of the group $\chi(G)\subseteq K^*$
    is not a prime power, then $\mathbf{O}=1$, see 
    Prop. \ref{p63}.
\end{ex}

\begin{ex}
    Assume that $K=K_0$ is the field of definition of $\chi$,
    and the order of the group $\chi(G)$
    is a prime power. Then $\mathbf{O}=2$, see 
    Prop. \ref{p64}. For instance, if $\chi(G)$ has order $2$, 
so that $K_0=\mathbb{Q}$, there are precisely two
conjugacy classes of representations, representatives of which
take the generator to the matrices $\sbmattrix 0110$ and 
$\sbmattrix 100{-1}$, respectively. 
\end{ex}

If $K$ strictly contains $K_0$, the coefficient 
$\mathbf{O}$ can be arbitrarily large in the prime power 
case. Example \ref{e67} in \S\ref{S6} shows a case where
$\mathbf{O}=4$.

Next we consider a representation $\varphi$ that is 
indecomposable over $K$ but becomes decomposable over some 
extension $L/K$, which can be assumed to be Galois. 
Over $L$, this representation has the form 
$\varphi=\sbmattrix{\chi_1}00{\chi_2}$, where $\chi_1$
and $\chi_2$ are one dimensional representation and 
the set $\{\chi_1,\chi_2\}$ is necessarily a Galois orbit.  
Over $K$ the image is contained in a two dimensional 
subalgebra $\mathfrak{L}$ of $\mathbb{M}_2(K)$ which
is necessarily a field. In particular, the image of $\varphi$
is generated by a matrix $\mathtt{r}_0$ whose eigenvalues are 
contained in a quadratic extension of $K$ which can be assumed 
to be $L$. Under this assumption, the algebra 
$\mathfrak{L}$ is isomorphic to the field extension
$L$, and we fix one such isomorphism in all that follows.
If $\mathfrak{H}_\varphi$ is the ring of integers
in $\mathfrak{L}$, i.e., the ring of matrices in
$\mathfrak{L}$ that are integral over $\mathcal{O}_K$,
it can be proved that the smallest order
$\mathcal{O}_K[\mathtt{r}_0]$ containing 
$\mathtt{r}_0$ has the form $\mathcal{O}_K[\mathtt{r}_0]=
\mathcal{O}_K\mathtt{1}+I_\varphi\mathfrak{H}_\varphi$,
where $\mathtt{1}$ is the identity matrix and
$I_\varphi\subseteq\mathcal{O}_K$ is an ideal that is
non-trivial only at places where $\mathtt{r}_0$ has
non-trivial side branches. Furthermore, the distance
ideal $I_{(v,w)}$ between the anchor vertex $v$
and any other vertex in the product $\mathbb{B}$
is a divisor of $I_\varphi$. See \S\ref{S9B}.
Recall that $L^*\cong\mathfrak{L}^*$ acts on the 
Bruhat-Tits tree via Moebius transformations. 

 \begin{Thm}\label{t6}
    Assume that the field $K$ does not contain
    the eigenvalues $\xi_1$ and $\xi_2$ of $\mathtt{r}_0$.
    The group $$U=\{\lambda\in 
    L^*|\lambda\mathcal{O}_L=J\mathcal{O}_L
    \textnormal{ for some }J\in\mathcal{I}_K(2)\}
    \subseteq L^*.$$  acts on $\mathbb{B}$ 
    preserving both the anchor vertex $v$ and the distance 
    ideal $I_{(v,w)}$ for any vertex $w\in\mathbb{B}$.  
    We let $\mathbf{O}_{I'}$ denote the 
    number of orbits of vertices $w$ satisfying
    $I_{(v,w)}=I'$. Furthermore, we let  $$U_{I'}=
    \left\{\lambda\in U\Big| \frac{\sigma(\lambda)}
    {\lambda}\equiv 1\mod I'I[L/K]\right\},$$
    where $I[L/K]$ is the different of the quadratic extension.
    Then the correspondence $I'\mapsto U_{I'}$ preserves
    inclusions and satisfies $U_{(1)}=U$. Furthermore,
    assume that $L$ is not contained in the Hilbert class field
    of $K$. Then the number of 
    $\mathrm{GL}_2(\mathcal{O}_K)$-conjugacy classes of 
    $\mathcal{O}_K$-representations of $G$ that are 
    $\mathrm{GL}_2(K)$-conjugate to $\varphi$ is 
    $$h_{L/K}\sum_{I'|I_\rho}\mathbf{O}_{I'}\frac{\mu_{(1)}}
    {\mu_{I'}},$$
    where $\mu_{I'}$ is the order of the image
    of $U_{I'}$ in $\mathcal{G}_K(2)$.
\end{Thm}

The condition that $L$ is not contained in the Hilbert class 
field of $K$ is there to ensure that the maximal order 
$\mathfrak{H}_\varphi$ of $K[\mathtt{r}_0]$ is not selective 
(c.f. Lemma \ref{lsel}). We refer to it as the selectivity condition.

\begin{ex}
      Let $\eta$ be a primitive $n$-th root of unity, 
      $L=\mathbb{Q}[\eta]$, and let
      $\chi$ be an invertible character sending the generator
      of a cyclic group $C$ to $\eta$. 
      Let $K=L^{\sigma}\subseteq L$,
      be the invariant field of an element 
      $\sigma$ of order two in the
      Galois group $\mathrm{Gal}(L/\mathbb{Q})
      \cong(\mathbb{Z}/n\mathbb{Z})^*$, and let $\varphi$ 
      be the representation with character $\chi+\sigma(\chi)$,
      which is defined over $K$. In this case,
      side branches are trivial, hence the number of conjugacy 
      classes of $\mathcal{O}_K$-representations that are 
      $K$-conjugate to $\varphi$ is precisely $h_{L/K}$.
      See Prop. \ref{p42}.
\end{ex}

\begin{ex}
    The cyclic group $C_3$ of order $3$ has a 
    unique irreducible rational 
    representation. We set $K=\mathbb{Q}$ and 
    $L=\mathbb{Q}[\omega]$,
    where $\omega$ is a primitive cubic root of $1$. 
    Then $\mathcal{O}_L$
    is a PID, so, acording to the preceding example, there 
    is a unique representations of $C_3$ over $\mathbb{Z}$.
\end{ex}

\begin{ex}\label{e25}
    Let $G=C_{2^n}$, let 
    $K=\mathbb{Q}[\zeta_{2^n}+\zeta_{2^n}^{-1}]$,
    and let $L=\mathbb{Q}[\zeta_{2^n}]$, where $\zeta_m$
    is a primitive $m$-th root of unity. Then it is known
    that the class number of $K$ is $1$ for $n\leq 8$ \cite{Mill14}.
    On the other hand, according to \cite[p.353]{wik1}, 
    for $n=6$ we have $h_L=17$, and for $n=7$ we have 
    $h_L=359.057$. In these cases we have $h_{L/K}=h_L$.
\end{ex}

Some cases where Theorem \ref{t6} applies, but non-trivial
side branches do appear, are shown in Example \ref{e911}
and Example \ref{e912}.
When the selectivity condition in Theorem \ref{t6} is not
satisfied, there are some vertices in $\mathbb{B}$
that fail to contribute to the total number of 
representations. Examples \ref{e89} and \ref{e89b}
illustrate this point.

Finally, we give a general result for totally 
indecomposable two dimensional representations $\psi$, 
i.e., faithful two dimensional representation of 
non-abelian groups. Here $\mathbb{S}(\psi)$
denote the product of the full local branches
$\mathfrak{s}_\nu(\psi)$ of
the representation, as opposite to the side
branches that appear in the preceding results.
Since the order generated by $\psi(G)$, in this case,
is maximal at almost all places, the local branch
is trivial, i.e., consist of a single vertex,
at almost all places. 

\begin{Thm}\label{t7}
Let $K$ be a number field.
Assume $\psi:G\rightarrow\mathrm{GL}_2(K)$ is an absolutely 
irreducible representation.
There is an $h_K(2)$ to $1$ map $f:\Psi\rightarrow\Omega$
from the set $\Psi$ of conjugacy classes of integral 
representations to the set $\Omega$ of maximal orders
containing $\psi$ that are isomorphic to 
$\mathbb{M}_2(\mathcal{O}_K)$.
Furthermore, the latter set is in correspondence
with the set of vertices in the product $\mathbb{S}(\psi)$ 
whose Artin distance to the distinguished vertex $v_0$ 
is trivial. Given one integral representation $\phi$ corresponding 
to a given vertex, the others can be obtained by conjugating
$\phi$ by a set of representatives of the quotient 
$\tilde{N}=N/K^*\mathrm{GL}_2(\mathcal{O}_K)$,
where $N$ is the normalizer of $\mathbb{M}_2(\mathcal{O}_K)$.
\end{Thm}

For the particular case of
dihedral groups, we have a very explicit result that
is obtained, as usual, by restricting the base field $K$
to the field of definition $K_0$.

\begin{Thm}\label{t4}
    Any faithful irreducible representation of a 
    dihedral group $D_n$ is 
    conjugate, over its field of definition $K_0$, 
    to precisely  
    $2h_K(2)$-representations over the ring of 
    integers $\mathcal{O}_K$
    that are pairwise non-conjugate, if $n$ is either 
    a prime power or
    twice a prime power. In any other case it is 
    conjugate to $h_K(2)$
    such representations.
\end{Thm}

All of our proofs employ local tools, 
centered around the Theory of Bruhat-Tits 
trees. This is especially well suited to study local integral 
representations, which is already an important topic on its own.
Indeed, some effort has been made by other authors to
describe such representations, see \cite{Man99} as an example. 
It is possible to give a 
proof of Theorem \ref{t4} using more standard methods from
representation theory, but the Bruhat-Tits point of view 
presented here make it easier to extend these computations
to larger fields, as should be more clear in the examples in \S\ref{S9}.
However, a direct proof can be useful to generalize this result
to higher dimensional representations of metacyclic groups.
We expect to return to this issue in a subsequent work.

\section{Ideals, extensions, lattices and orders}

Recall that, for every finite place $\nu$ of $K$, the completions of $K$ and $\mathcal{O}_K$
are denoted $K_{\nu}$ and $\mathcal{O}_{\nu}$, respectively. However, When recalling the
field $K$ is important, we use $\mathcal{O}_{K,\nu}$ instead of 
$\mathcal{O}_{\nu}$. Note that $\mathcal{O}_{K,\nu}=
\mathcal{O}_{L,\omega}\cap K_{\nu}$
whenever $\omega$ is a place of $L$ lying over $\nu$. 
Similar conventions apply
to the notation $I_{\nu}$ for the completion of a fractional ideal $I$.
We also write $K_{\nu}$
for infinite (archimedean) places of $K$, although $\mathcal{O}_{\nu}$
and $I_{\nu}$ fail to be defined in that case. 
The maximal ideal $P=P[\nu]$ corresponding to $\nu$ satisfies
$P_{\nu}=\pi_{\nu}\mathcal{O}_{\nu}$, where $\pi_{\nu}$ is a 
local uniformizer
at $\nu$, and $P_{\nu'}=\mathcal{O}_{\nu'}$ elsewhere. 
Recall that a finite 
extension $L/K$ is said to be unramified  at a place $\nu$
if the ideal $P'=P[\omega]\subseteq\mathcal{O}_L$ satisfies
 $P'_\omega=\mathcal{O}_{L,\omega}P_{\nu}$ for every 
 place $\omega$ of $L$ lying 
 above $\nu$. At unramified places we have 
 $\mathcal{O}_{L,\omega}=\mathcal{O}_{K,\nu}[u_1,\dots,u_m]$
for any set of elements whose images $\bar{u}_1,\dots,\bar{u}_m$
in the residue field $\mathbb{L}_\omega=
\mathcal{O}_{L,\omega}/\pi_\omega\mathcal{O}_{L,\omega}$
satisfy $\mathbb{L}_\omega=\mathbb{K}_{\nu}[\bar{u}_1,\dots,\bar{u}_m]$,
where $\mathbb{K}_{\nu}=\mathcal{O}_{K,\nu}/\pi_{\nu}\mathcal{O}_{K,\nu}$.
 Similarly, $L/K$ is said to be fully ramified  
 at $\nu$ if both of the following conditions hold:
 \begin{itemize}
     \item There is a unique place $\omega$ over $\nu$, and
     \item $\mathcal{O}_{L,\omega}P_{\nu}=(P'_\omega)^{[L:K]}$.
 \end{itemize}
  In this case we have
 $\mathcal{O}_{L,\omega}=\mathcal{O}_{K,\nu}[\pi_\omega]$ for any local 
 uniformizer $\pi_\omega\in\mathcal{O}_{L,\omega}$. 
 This is proved in detail in \cite[\S 16:4]{O73}.
Also note that, when writing $\nu=\nu_P$
 for the valuation corresponding to a maximal ideal $P$,
 we use notations such as
 $K_P$ or $\pi_P$ instead of the cumbersome
 $K_{\nu_P}$ or $\pi_{\nu_P}$. This is used mostly in 
 examples.

For any finite dimensional $K$-vector space $V$, we define its 
completion at $\nu$ by $V_{\nu}=K_{\nu}\otimes_KV$. If $\{\mathbf{v}_1,\dots
\mathbf{v}_n\}$ is a basis of $V$ over $K$, any $\mathcal{O}_K$-module
$\Lambda$ satisfying 
$$a\sum_{i=1}^n\mathcal{O}_K\mathbf{v}_i\subseteq\Lambda
\subseteq b\sum_{i=1}^n\mathcal{O}_K\mathbf{v}_i,$$
for some $a,b\in K$,
is called a full lattice, or more precisely, a full
$\mathcal{O}_K$-lattice on $V$. Note that this definition 
is independent of the choice of a basis. A full lattice
in a subspace of $V$ is called simply a lattice in $V$.
When $V$ is a $K$-algebra, an $\mathcal{O}_K$-order in $V$
is an $\mathcal{O}_K$-lattice that is also a subring.
Full orders are defined analogously. If $K\subseteq L$
are two number fields, then $\mathcal{O}_L$ is a full 
$\mathcal{O}_K$-order in $L$. Every $\mathcal{O}_K$-lattice
in $K$ is a fractional ideal.

For any lattice $\Lambda$ in a vector space $V$, we denote
by $\Lambda_{\nu}\subseteq V_{\nu}$ the completion at $\nu$, which is actually
an $\mathcal{O}_{\nu}$-lattice in $V_{\nu}$. A full lattice $\Lambda$ on $V$
is completely determined by the family $\{\Lambda_{\nu}\}_{\nu}$ of
its completions at all finite places. Furthermore, every
such family satisfying a coherence relation determines 
a full lattice.
To be precise, if we fix an arbitrary lattice $\Lambda$, 
any family
$\{\Lambda''(\nu)\}_{\nu}$ of local lattices, with 
$\Lambda''(\nu)\subseteq V_{\nu}$, is called coherent 
with $\Lambda$ if
$\Lambda''(\nu)=\Lambda_{\nu}$ for every $v$ outside a finite 
exceptional set.  This concept is important because
it has the following list of properties, to which we refer
as the local global principle for lattices (\textbf{LGPL}):
\begin{enumerate}[label=($\clubsuit$\bf{\arabic*})]
\item\label{it1}
If $\Lambda$ and $\Lambda'$ are two lattices, then
    $\Lambda_{\nu}=\Lambda'_{\nu}$ for every $\nu$ outside 
    a finite set of exceptional non-archimedean places.
    In particular, the coherence property is independent of the
    choice of a particular lattice $\Lambda$, so we just speak of 
    coherent families.
\item\label{it2}
If $\Lambda_{\nu}=\Lambda'_{\nu}$ for every $\nu$, 
 i.e., the set of 
 exceptional places is empty, then we have $\Lambda=\Lambda'$.
\item \label{it3}
If $\{\Lambda''(\nu)\}_{\nu}$ is a coherent family of lattices,
 then there exists a lattice $\Lambda''$ satisfying
 $\Lambda''(\nu)=\Lambda''_{\nu}$ at every non-archimedean place.
\end{enumerate}
The last property above allows us to modify an existing lattice
at a finite set of places to define a new lattice. 
See \cite[\S81:14]{O73} for details.
These properties allows us to study lattices locally, one place at a time.
One important example of this ``local modification'' is the following:
Let $a$ be an adelic linear transformation, i.e., an element
$a=(a_\nu)_\nu\in\prod_\nu\mathrm{Aut}_{K_\nu}(V_\nu)$
for which all coordinates outside a finite set are
in the stabilizer $\mathrm{Aut}_{\mathcal{O}_\nu}(\Lambda_\nu)$
of $\Lambda_\nu$. Then the adelic image $a\Lambda$
is defined by modifying the lattice at the remaining places.
This defines an action of adelic matrices on lattices
that is assumed in all that follows.
It can be used for fractional ideals and orders,
in the latter assuming that automorphisms are ring automorphisms.
Since the local ring $\mathcal{O}_{\nu}$ is a discrete valuation domain,
and therefore also a principal ideal domain, local study allows us to use
properties of PIDs even when the class group of the global field
$K$ is not trivial. In fact, every fractional ideal, principal or not,
has an adelic generator.
Lattices that differ at a single place $\nu$
are called $\nu$-variants.

Next result, which is well known, is critical in all that follows. 
For a proof, see \cite[\S21:3]{O73}.

\begin{Lem}\label{lap1}
\textbf{(Strong Approximation Theorem).}
For every finite set of non-archimedean places $\nu_1,\dots,\nu_n$,
every family $\{a_i\}_{1\leq i\leq n}$ with $a_i\in K_{\nu_i}$ and
every real number $M$, there exists an element
$a\in K$ such that $\nu_i(a_i-a)>M$, for $1\leq i\leq n$,
and $\omega(a)\geq0$ for every non-archimedean place
$\omega\notin \{\nu_1,\dots,\nu_n\}$.\qed
\end{Lem}

In fact, we require an approximation result for matrices.
This is also well known, but we use the explicit proof outlined 
here  to actually compute examples.
Recall that, for a matrix $\mathtt{T}\in\mathrm{GL}_2(K_{\nu})$, 
the valuation  $\nu(\mathtt{T})$ is defined as the minimum of the 
valuations of the coefficients.
Recall also that we use the convention $\nu(0)=\infty$.

\begin{Lem}\label{lap2}
\textbf{(Strong Approximation for $SL_2$).}
For every finite set of non-archimedean places $\nu_1,\dots,\nu_n$,
every family $\{\mathtt{T}_i\}_{1\leq i\leq n}$ with 
$\mathtt{T}_i\in \mathrm{SL}_2(K_{\nu_i})$ and
every real number $M$, there exists an element
$\mathtt{T}\in \mathrm{SL}_2(K)$ such that 
$\nu_i(\mathtt{T}_i-\mathtt{T})>M$, for $1\leq i\leq n$,
and $\omega(\mathtt{T})\geq0$ 
for every non-archimedean place
$\omega\notin \{\nu_1,\dots,\nu_n\}$.
\end{Lem}

\begin{proof}
This is immediate from the preceding results and the formulas
\begin{equation}\label{mat1}
\sbmattrix abcd=\sbmattrix10{(d-1)/b}1
\sbmattrix1b01
\sbmattrix10{(a-1)/b}1,
\qquad
\sbmattrix a0cd=\sbmattrix1{-1}01\sbmattrix {a+c}dcd,
\end{equation}
when $b\neq 0$. You can reduce any approximation problem 
to approximate one of the $\mathtt{T}_i's$, at one place,
and the identity at  the remaining places.
Note that we have the identity $cb=da-1$ on the left, 
as the determinant of the matrix equals $1$.
\end{proof}

  For any lattice $\Lambda$ in an $n$-th dimensional vector space $V$,
  we write $\bigwedge^k\!\! \Lambda\subset\bigwedge^k\!\! V$
  for the lattice generated by the wedge products of the form
  $\mathbf{m}_1\wedge\cdots\wedge \mathbf{m}_n$ with
  $\mathbf{m}_1,\dots,\mathbf{m}_n\in\Lambda$.
  In particular, note that for any basis 
  $\{\mathbf{x}_1,\dots,\mathbf{x}_n\}$, and for arbitrary
  fractional ideals $I_1,\dots, I_n$,
  we have the identity $\bigwedge^k\!
  (I_1\mathbf{x}_1+\cdots+I_n\mathbf{x}_n)=
  I_1\cdots I_n(\mathbf{x}_1\wedge\cdots\wedge \mathbf{x}_n)$.

\begin{Lem}\label{lap3}
    Every full
    $\mathcal{O}_K$-lattice $\Lambda$ on a two-dimensional vector
    space $V$ over $K$ is isomorphic, as an $\mathcal{O}_K$-module, 
    to the lattice $\mathcal{O}_K\times I\subseteq K^2$,
    for some fractional ideal $I\subseteq K$.
    Furthermore, we can find $I$ by the formula
    $\bigwedge^{2}\Lambda=I(\mathbf{e}_1\wedge\mathbf{e}_2)
    \subseteq\bigwedge^{2}V$, where $\{\mathbf{e}_1,\mathbf{e}_2\}$
    is a basis of $V$.
\end{Lem}

\begin{proof}
Define $I$ by the identity
    $\bigwedge^{2}\Lambda=I(\mathbf{e}_1\wedge\mathbf{e}_2)$.
Let $\Lambda'=\mathcal{O}_K\mathbf{e}_1\oplus I\mathbf{e}_2$.
Then we can find a local matrix $\mathtt{T}_{\nu}\in 
\mathrm{GL}_2(K_{\nu})$
such that $\mathtt{T}_{\nu}\Lambda_{\nu}=\Lambda'_{\nu}$, 
where $\mathtt{T}_{\nu}$ is the identity
for almost every place $\nu$. The fact that 
$\bigwedge^{2}\Lambda=\bigwedge^{2}\Lambda'$ implies that the
determinant of $\mathtt{T}_{\nu}$ is a unit locally everywhere, and can be 
assumed to be one, since it is easy to write down an explicit
matrix fixing $\Lambda'$ whose determinant is a prescribed unit.
Now the result follows from the preceding lemma, since the stabilizer
of a lattice in $\mathrm{GL}_2(K_{\nu})$ is open.
\end{proof}

Note that Lemma \ref{lap2} generalizes effortlessly to 
arbitrary dimension,
since $\mathrm{SL}_n(K_{\nu_i})$ can be generated using 
elementary matrices, which move at most two basis vectors at 
a time. This implies the following generalization of 
Lemma \ref{lap3}, whose proof we skip (see \cite[\S81:5]{O73}):

\begin{Lem}\label{le1}
    Every non-trivial
    $\mathcal{O}_K$-lattice in a finitely dimensional vector
    space over $K$ is isomorphic, as an $\mathcal{O}_K$-module, to the
    lattice $\mathcal{O}_K\times\cdots\times\mathcal{O}_K\times I
    \subseteq K^n$,
    for some integer $n>0$, and for some fractional ideal $I\subseteq K$.
    \qed
\end{Lem}

In the examples we use the convention of writing 
$\mathbf{p}_1,\dots,\mathbf{p}_n$ for the maximal ideals in 
the ring of integers $D=\mathcal{O}_K$ of a number field $K$
containing a given rational prime $p$.
For instance, $\mathbf{2}_1,\dots,\mathbf{2}_n$ denote the dyadic
maximal ideals. Recall that notations like
$\Lambda_{\mathbf{2}_1}$ are used instead of
$\Lambda_{\nu_{\mathbf{2}_1}}$.
To simplify computations, we also write $\sbvecttor IJ$
for the lattice $I\times J$, and ${\sbvecttor IJ}_\nu$ or
${\sbvecttor IJ}_{\mathbf{p}_i}$ for its completion
at $\nu=\nu_{\mathbf{p}_i}$.
\begin{ex}\label{ex1}
    Consider the ideal $\mathbf{2}_1=(2,1+\sqrt{-5})$ in the ring
    $D=\mathbb{Z}[\sqrt{-5}]$. This is a non-principal ideal satisfying
    $\mathbf{2}_1^2=2D$. Then the lattice 
    $\sbvecttor{\mathbf{2}_1}{\mathbf{2}_1}$ is free.
    In fact, the previous lattices and $\sbvecttor{D}{D}$ have a different
    completion only at the dyadic place, and we can write
    $${\sbvecttor{\mathbf{2}_1}{\mathbf{2}_1}}_{\mathbf{2}_1}=
     {\sbvecttor{\left(1+\sqrt{-5}\right)D}{\left(1+\sqrt{-5}\right)D}}_{\mathbf{2}_1}=
     \sbmattrix{1+\sqrt{-5}}00{1+\sqrt{-5}}
     {\sbvecttor DD}_{\mathbf{2}_1},$$
    and therefore
    $$
    {\sbvecttor D{2D}}_{\mathbf{2}_1}=
     \sbmattrix1002\sbmattrix{-2+\sqrt{-5}}001
     \sbmattrix{\frac1{1+\sqrt{-5}}}00{\frac1{1+\sqrt{-5}}}
     {\sbvecttor{\mathbf{2}_1}{\mathbf{2}_1}}_{\mathbf{2}_1}$$
     $$=\sbmattrix{\frac{1+\sqrt{-5}}2}00{\frac2{1+\sqrt{-5}}}
     {\sbvecttor{\mathbf{2}_1}{\mathbf{2}_1}}_{\mathbf{2}_1}
     .
    $$
    The second factor on the right hand side of the first line
    has no effect on the lattice, since $-2+\sqrt{-5}$ is a unit at
    the dyadic place, and it is only there to ensure that 
    the determinant of the product is $1$. 
    Using Equation (\ref{mat1}), we can write
    $$\sbmattrix{\frac{1+\sqrt{-5}}2}00{\frac2{1+\sqrt{-5}}}=
    \sbmattrix1{-1}01\sbmattrix10{\frac{1-\sqrt{-5}}2}1
    \sbmattrix1{\frac2{1+\sqrt{-5}}}01
    \sbmattrix10{\frac{-3}2}1.
    $$
    We need to approximate these unipotent matrices by others 
    having the same effect on the local lattice at the dyadic place, 
    while having integral coefficients at non-dyadic places. The only 
    coefficient that fails to satisfy the latter
    condition is $$\frac2{1+\sqrt{-5}}=\frac13(1-\sqrt{-5})=
    (\sqrt{-5}-1)(1+4+4^2+\dots).$$
    The matrix $\sbmattrix10{\frac{-3}2}1$ takes 
    $\sbvecttor{\mathbf{2}_1}{\mathbf{2}_1}$ to the 
    lattice $\mathbf{2}_1\Lambda$,
    where $\Lambda\subseteq K_{\mathbf{2}_1}$ is generated by 
    the vectors $\sbvecttor{1}{-3/2}$ and $\sbvecttor01$. 
    Hence, any matrix of the form
    $\sbmattrix1{4m}01$ fixes that lattice. Approximating
    $\frac2{1+\sqrt{-5}}$ by $(\sqrt{-5}-1)$ is therefore sufficient.
    This leads to the approximation
    $$\mathtt{T}=
    \sbmattrix1{-1}01\sbmattrix10{\frac{1-\sqrt{-5}}2}1
    \sbmattrix1{-1+\sqrt{-5}}01
    \sbmattrix10{\frac{-3}2}1$$
    $$=\sbmattrix{\frac{13+\sqrt{-5}}2}{-4}{-4-2\sqrt{-5}}{3+\sqrt-5}
    ={\sbmattrix{3+\sqrt-5}{4}{4+2\sqrt{-5}}{\frac{13+\sqrt{-5}}2}}^{-1}.$$
    By the preceding argument, we have 
    $\mathtt{T}{\sbvecttor{\mathbf{2}_1}{\mathbf{2}_1}}_{\mathbf{2}_1}=
    {\sbvecttor{D}{2D}}_{\mathbf{2}_1}$, and the same holds at 
    non-dyadic places since both $\mathtt{T}$ and its inverse have integral coefficients there.
    We conclude that $\sbvecttor{\mathbf{2}_1}{\mathbf{2}_1}$ is 
    generated by the vectors $\sbvecttor{3+\sqrt{-5}}{4+2\sqrt{-5}}$ 
    and $\sbvecttor{8}{13+\sqrt{-5}}$.
\end{ex}

Now let $L/K$ be a finite extension, and let $I\subseteq L$ be
a fractional ideal. Then, as $L$ is a $K$-vector space, $I$ is
an $\mathcal{O}_K$-lattice, and therefore is isomorphic as such
to an $\mathcal{O}_K$-lattice of the form
$\mathcal{O}_K\times\cdots\times\mathcal{O}_K\times I'$,
for some fractional ideal $I'\subseteq K$.

\begin{Lem}\label{lp1}
    In the ideal $I'$ is as in the preceding paragraph, we have
    $I'=N_{L/K}(I)J$, for some ideal $J=J(L/K)\subseteq K$
    whose class is independent of $I$.
\end{Lem}

\begin{proof}
  Assume $\Lambda$ is a $\mathcal{O}_K$-lattice in $L$.
  Then it is immediate from the definition of the wedge 
  product of a lattice that
\begin{equation}\label{wnorm}
  \bigwedge^{[L:K]}\! a\Lambda=N_{L/K}(a)\bigwedge^{[L:K]}\! \Lambda, 
\end{equation}
  for any $a\in L$. Now write $I=a\mathcal{O}_L$ for some
  adelic 1-by-1 matrix $a$. The local version of
  Equation (\ref{wnorm}) gives us 
  $\bigwedge^{[L:K]}\! I=N_{L/K}(I)\bigwedge^{[L:K]}\! \mathcal{O}_L$,
  whence it suffices to define $J$ so that
  $J(\mathbf{x}_1\wedge\dots\wedge \mathbf{x}_{[L:K]})
  =\bigwedge^{[L:K]}\! \mathcal{O}_L$
  for one fix basis $\{\mathbf{x}_1,\dots,\mathbf{x}_{[L:K]}\}$ of $L$ 
  over $K$. It is immediate that the class of $J$ 
  is independent from the choice of the basis.
\end{proof}

\section{Branches and representations}\label{S4B}

We are often interested, in Number Theory, 
in classifying some type of  integral structure. 
These structures can be often linked to the concept 
of lattice or order. Although frequently omitted, 
the main tool to relate these concepts is
the following lemma from group actions:

\begin{Lem}\label{mainlem}
    Let $X$ and $Y$ be $G$-sets with $G$ acting transitively on $X$. 
    Let $Z\subseteq X\times Y$
    be a $G$-invariant subset. For every element $x\in X$, let $G_x=\mathrm{Stab}_G(x)$
    and let $Y_x=\{y\in Y|(x,y)\in Z\}$. Then, there is a natural bijective correspondence $\phi$
    between the orbit sets $G_x\backslash Y_x$ and $G\backslash Z$ sending the class 
    of $y\in Y_x$ to the class of $(x,y)\in Z$.
\end{Lem}

\begin{proof}
    If $y\in Y_x$ and $g\in G_x$, then 
    $g \cdot (x,y)=(g \cdot x , g \cdot y )=(x,g \cdot y)\in Z$, 
    hence $Y_x$ is $G_x$-invariant.
    Denote by $[y]$ the $G_x$-orbit of $y\in Y_x$, 
    and by $[x,y]$ the $G$-orbit of 
    $(x,y)\in Z$. If $[y]=[y']$, then $y'=g\cdot y$ for 
    some $g\in G_x$, whence $[x,y']=[x,y]$
    as before. It follows that $\phi$ is well defined. 
    If $[x,y']=[x,y]$, there must exist
    $g\in G$ satisfying $g\cdot (x,y)=(x,y')$, 
    so in particular $g\cdot x=x$, and $g\in G_x$. This
    proves that $\phi$ is injective. For the surjectivity, 
    we observe that, for any element
    $x'\in X$, by the transitivity hypotheses, 
    there exists $g\in G$ such that $g\cdot x=x'$.
    We conclude that any class of the form $[x',y']$ 
    can be written as $[x',y']=
    [g\cdot x,g\cdot y]=[x,y]$, where $y=g^{-1}\cdot y'$.
    In particular, since $Z$ is $G$-invariant, if $(x',y')\in Z$, then
    $(x,y)\in Z$, whence $y\in Y_x$. The result follows.
\end{proof}

\begin{ex}
    If $X=Y$, and $Z$ is the diagonal, then $Y_x=\{x\}$. On the other hand
    $G\backslash Z$ is trivial, since $G$ acts transitively on $X=Y$,
    and therefore also on the diagonal.
\end{ex}

\begin{Cor}\label{maincor}
    Let $X$ and $Y$ be transitive $G$-sets.
    Let $Z\subseteq X\times Y$
    be a $G$-invariant subset. Fix elements $x\in X$ and $y\in Y$ 
    with $(x,y)\in Z$. Let $G_x=\mathrm{Stab}_G(x)$, 
    $G_y=\mathrm{Stab}_G(y)$, $Y_x=\{y\in Y|(x,y)\in Z\}$ and
    $X_y=\{x\in X|(x,y)\in Z\}$. 
    Then, there is a natural bijective correspondence $\phi$
    between the orbit sets $G_x\backslash Y_x$ and 
    $G_y\backslash X_y$. \qed
\end{Cor}

\begin{proof}[Proof of Thm. \ref{percc}]
 Let $X$ be the 
 $\mathrm{GL}_2(K)$-conjugacy class of the
 representation $\phi_0$. Let $Y$ be the 
 $\mathrm{GL}_2(K)$-conjugacy class of the 
$\mathcal{O}_K$-orders 
$\mathbb{M}_2(\mathcal{O}_K)\subseteq
\mathbb{M}_2(K)$. Both are transitive
$\mathrm{GL}_2(K)$-sets by definition. Finally, 
define $$Z=\{(\phi,\mathfrak{D})|\phi(g)\in\mathfrak{D}^*,
\forall g\in G\}.$$
Consider the order $\mathfrak{D}_0=
\mathbb{M}_2(\mathcal{O}_K)$, 
and the pair $(\phi_0,\mathfrak{D}_0)\in Z$.
It follows from Cor. \ref{maincor} that 
there is a one to one correspondence 
between the set $\hat{\Phi}$ of 
$N$-conjugacy classes in $\tilde{\Phi}$, 
and the set $\Omega$. It suffices, 
therefore to see that the  set of
$\mathrm{GL}_2(\mathcal{O}_K)$-conjugacy  
classes of integral representations in a 
suitable given $N$-conjugacy class is in 
correspondence with $\tilde{N}_\mathtt{a}$. 
Since $\mathrm{GL}_2(\mathcal{O}_K)$ is 
normal in $N$, and since the quotient 
$N/\mathrm{GL}_2(\mathcal{O}_K)$
is abelian, the group 
$\tilde{N}_\mathtt{a}$ coincide with
$\tilde{N}_{\mathtt{n}\mathtt{a}}$ for 
$\mathtt{n}\in N$. 
    
 In fact, assume that $\phi$ and $\phi'$ 
 are representations in $\tilde{\Phi}$ 
 satisfying
 $\phi'=\mathtt{n}\phi \mathtt{n}^{-1}$ for 
 some $\mathtt{n}\in N$. Furthermore,
 assume $\phi=
 \mathtt{a}\phi_0\mathtt{a}^{-1}$, so that
the pair $(\phi,\mathfrak{D}_0)$ is in the 
same class as $(\phi_0,\mathfrak{D})$, 
where 
$\mathfrak{D}=\mathtt{a}^{-1}
\mathfrak{D}_0 \mathtt{a}$,
as above. Then both, $\phi$ and $\phi'$,
correspond to the class of the
order $\mathfrak{D}$. Then they are in the 
same $\mathrm{GL}_2
(\mathcal{O}_K)$-conjugacy class if 
and only if there exists
an element $\mathtt{g}\in
\mathrm{GL}_2(\mathcal{O}_K)$ satisfying 
$\phi'=\mathtt{g}\phi \mathtt{g}^{-1}$. In 
particular, this implies that $\mathtt{l}=
\mathtt{g}^{-1}\mathtt{n}$ centralizes 
$\phi$, whence it belongs to the 
centralizer 
$\mathtt{a}\mathfrak{L}^*\mathtt{a}^{-1}$ 
of $\phi$.  Conversely, if we can write 
$\mathtt{n}=\mathtt{g}\mathtt{l}$, with
$\mathtt{l}\in(N\cap 
\mathtt{a}\mathfrak{L}^*\mathtt{a}^{-1})$
and $\mathtt{g}\in
\mathrm{GL}_2(\mathcal{O}_K)$, then
$\phi$ and $\phi'$ are in the same 
$\mathrm{GL}_2(\mathcal{O}_K)$-conjugacy 
class. We conclude that 
$\mathtt{n}\in\mathrm{GL}_2(\mathcal{O}_K)
(N\cap\mathtt{a}\mathfrak{L}^*
\mathtt{a}^{-1})$ precisely when $\phi$ 
and $\phi'$ are  
$\mathrm{GL}_2(\mathcal{O}_K)$-conjugates, 
and the result follows.
\end{proof}

\begin{Lem}\label{Lemcc}
    In the notations of Thm. \ref{percc}, we have
    $\tilde{N}_\mathtt{a}\cong\mathcal{G}_K(2)/\Theta_{\mathfrak{D}}$,
     where 
     $$\Theta_{\mathfrak{D}}=
     \left\{\overline{J}\in\mathcal{G}_K(2)\Big|J^2=
     \big(\det(c)\big)\textnormal{ for some } 
     c\in \mathfrak{L}^*\cap N_{\mathfrak{D}}\right\}.$$ 
\end{Lem}

\begin{proof}
    Recall that the map $\Lambda\mapsto\mathfrak{D}_\Lambda$,
    sending a lattice to its endomorphism ring, satisfies
    $\mathfrak{D}_\Lambda=\mathfrak{D}_{\Lambda'}$
    if and only if $\Lambda'=J\Lambda$, where $J$ is a fractional 
    ideal in the field $K$. It follows that a matrix $b$
    is in the normalizer $N=N_{\mathfrak{D}_\Lambda}$ of 
    $\mathfrak{D}_\Lambda$ precisely when 
    $b\Lambda=J\Lambda$ for $J$ as above. We apply this to the
    particular case $\Lambda=\Lambda_0=\mathcal{O}_K^2$. Then
    the correspondence $b\mapsto J$ defines a homomorphism
    $N\rightarrow \mathcal{G}_K$ whose kernel is 
    $K^*\mathfrak{D}_0^*$. It follows that there is 
    a matrix $b$
    corresponding to a fractional ideal $J$ precisely when
    $J\times J$ is free, which is equivalent to $J^2$ being 
    principal, according to Lemma \ref{lap3}. 
    This proves that the image is $\mathcal{G}_K(2)$.
    The result follows if we note that
    $b\Lambda=J\Lambda$ implies $\big(\det(b)\big)=J^2$,
    and the group $N\cap a\mathfrak{L}^*a^{-1}$ 
    has the same set of determinants
    as  $a^{-1}Na\cap \mathfrak{L}^*=
    N_{\mathfrak{D}}\cap \mathfrak{L}^*$.    
\end{proof}

Next section is devoted to describe the set $\Omega$, from 
Theorem \ref{percc}, in a graph-theoretical setting. 
To accomplish this, we recall there the main facts on the
theory of branches.

\section{Local lattices and trees}\label{S5B}

The results listed in this section can be found in more detail in
\cite{A13}, \cite{AAC18} and \cite{AB19}.

Set $V_{\nu}=K_{\nu}^2$. One way to describe all local 
lattices for this space
is by studying the Bruhat-Tits tree (BTT) 
$\mathfrak{t}=\mathfrak{t}_\nu$ defined in the
introduction. In this tree, every vertex is written 
in the form $v_a^{[n]}$,
for $a\in K_\nu$ and $n\in\mathbb{Z}$, and
can be described by any of the following objects:
\begin{enumerate}[label=($\heartsuit$\bf{\arabic*})]
    \item\label{ballenc} The ball $B_a^{[n]}=\{x\in K_\nu|
    \nu(x-a)\geq n\}\subseteq K_\nu$.
    \item The homothety class of the lattice 
    $\Lambda_a^{[n]}=
    \left\langle\sbvecttor a1,\sbvecttor{\pi_{\nu}^n}0\right\rangle
    \subseteq K_{\nu}^2$.
    \item The maximal order 
    $\mathfrak{D}_a^{[n]}=
    \mathrm{End}_{\mathcal{O}_{\nu}}\left(\Lambda_a^{[n]}\right)\subseteq
    \mathbb{M}_2(K_{\nu})$.
\end{enumerate}
In particular, $v_a^{[n]}=v_b^{[m]}$ if and only if
$n=m$ and $\nu(b-a)\geq n$.
Note that every homothety class of lattices in $V_{\nu}$ contains a
unique lattice of the form $\Lambda_a^{[n]}$. For any ball
$B_a^{[n]}$, the neighbors of $v_a^{[n]}$ are the vertices
corresponding to the proper maximal sub-balls of $B_a^{[n]}$,
or to the minimal over-ball $B_a^{[n-1]}$. Alternatively,
two maximal orders, $\mathfrak{D}$ and $\mathfrak{D}'$,
correspond to neighboring vertices if they have, in some basis,
the form
\begin{equation}\label{nord}
\mathfrak{D}=
\sbmattrix{\mathcal{O}_{\nu}}{\mathcal{O}_{\nu}}
{\mathcal{O}_{\nu}}{\mathcal{O}_{\nu}},
\qquad
\mathfrak{D}'=
\sbmattrix{\mathcal{O}_{\nu}}{\pi^{-1}_{\nu}\mathcal{O}_{\nu}}
{\pi_{\nu}\mathcal{O}_{\nu}}{\mathcal{O}_{\nu}}.\end{equation}
An analogous characterization can be given in terms of lattices
by carefully choosing representatives. 
 The graph thus defined turns out to be a tree with 
 every vertex having $p+1$ neighbors, where $p$ is the 
 cardinality of the residue field $\mathbb{K}_\nu$. This is easier to see using the description in terms of balls.
 Recall that neighbors in a tree are precisely
the vertices at distance one, if the distance between two
vertices is defined as the number of edges of the smallest
subtree containing both.
The description in terms of balls in \ref{ballenc}
makes easy to interpret the visual limits of the tree as the
$K_{\nu}$-points of the projective line
$\mathbb{P}^1$. Recall that a visual limit of a tree is an equivalent class
of rays, where two rays are equivalent whenever their intersection is also 
a ray. In pictures, we use throughout the convention of
denoting visual limits as stars $(\star)$.

The vertex corresponding to a ball $B$ or an order $\mathfrak{D}$ is denoted
$v_B$ or $v_{\mathfrak{D}}$ when appropriate.
There is a natural action of $\mathrm{PSL}_2(K_{\nu})$ on the 
Bruhat-Tits tree that can be
described in a number of ways, which can be proven to be equivalent. The action
of the element $\mathtt{g}=\sbmattrix abcd$, associated to the
Moebius transformation $\gamma(z)=\frac{az+b}{cz+d}$,
on the vertex $v=v_B=v_{\mathfrak{D}}$ 
is computed by any of the following equivalent
procedures (see \cite[\S4]{AAC18} and \cite[\S3]{AB24} for details):
\begin{enumerate}[label=($\spadesuit$\bf{\arabic*})]
    \item $\mathtt{g}*v=v_{\mathtt{g}\mathfrak{D}\mathtt{g}^{-1}}$.
    \item If $\mathfrak{D}=\mathrm{End}_{\mathcal{O}_\nu}(\Lambda)$,
    then $\mathtt{g}*v=v_{\mathfrak{D}'}$, where
    $\mathfrak{D}'=\mathrm{End}_{\mathcal{O}_\nu}(\mathtt{g}\Lambda)$.
    \item Removing the vertex $v=v_B$ splits the set $\mathbb{P}^1(K_{\nu})$
    in $p+1$ sets, $p$ disjoint balls $B_1,\dots,B_p\subseteq B$ and a 
    ball complement $B^c$. They correspond to the set of visual limits
    of each connected component in the graph obtained by removing
    the vertex $v$ from the tree. We call $\{B_1,\dots,B_n,B^c\}$
    the partition defined by $B$. Then $\mathtt{g}*v=v_{B'}$, when
    $\{\gamma(B_1),\dots,\gamma(B_n),\gamma(B^c)\}$ is the partition 
    defined by the ball $B'$.
    \item The pointwise image $\gamma(B)$ of the ball $B$ is either a ball
    or the complement of a ball. Then $\mathtt{g}*v=v_{B'}$, where
    $B'=\gamma(B)$ if it is a ball, while, otherwise, $B'$ is
    the smallest ball properly containing $\gamma(B)^c$.
    \item\label{sic} Find three different  elements $z_1,z_2,z_3\in\mathbb{P}^1(K_{\nu})$, 
    so that $v=v_B$ is the incenter of the triplet, i.e., 
    the only vertex in each path connecting a pair of them. 
    Alternative, $B$ is the smallest ball containing at least two 
    points of the three. Then $\mathtt{g}*v$ is the incenter of
    $\gamma(z_1),\gamma(z_2)$ and $\gamma(z_3)$.
\end{enumerate}
The last procedure is often optimal for explicit computations.
See \cite{AB24} for an instance.

Let $\mathfrak{L}\subseteq\mathbb{M}_2(K_{\nu})$ be a maximal 
semisimple commutative subalgebra. Then the set
of maximal orders containing the ring of integers
$\mathcal{O}_\mathfrak{L}$ is 
the set of vertices in the subtree 
$\mathfrak{s}(\mathcal{O}_\mathfrak{L})$, 
called the branch of the order $\mathcal{O}_\mathfrak{L}$,
which is described, by cases, as follows (see \cite{A13}):
\begin{enumerate}[label=($\diamondsuit$\bf{\arabic*})]
    \item If $\mathfrak{L}/K_{\nu}$ is a ramified 
    field extension, then 
    $\mathfrak{s}(\mathcal{O}_\mathfrak{L})$ is an edge.
    \item If $\mathfrak{L}/K_{\nu}$ is an unramified 
    field extension, then 
    $\mathfrak{s}(\mathcal{O}_\mathfrak{L})$ is a vertex.
    \item If $\mathfrak{L}\cong K_{\nu}\times K_{\nu}$, then 
    $\mathfrak{s}(\mathcal{O}_\mathfrak{L})$ is a 
    maximal path, i.e.,
    a subgraph whose underlying topological 
    space is homeomorphic to the real line.
\end{enumerate}
Note that the visual limit of the latter 
branch has precisely two points,
and for every two points in the visual limit 
of $\mathfrak{t}_\nu$
there is a unique maximal path with those visual limits called
the maximal path connecting them.

More generally, an arbitrary full order $\mathfrak{H}\subseteq
\mathfrak{L}$ has the form 
$\mathfrak{H}=\mathcal{O}_\mathfrak{L}^{[t]}=
\mathcal{O}_{\nu}+\pi_{\nu}^t\mathcal{O}_\mathfrak{L}$, 
for some non-negative integer $t$, and the set of
maximal orders containing such order 
corresponds to a $t$-tubular
neighborhood of $\mathfrak{s}(\mathcal{O}_\mathfrak{L})$,
i.e., the maximal  tree 
$\mathfrak{s}(\mathcal{O}_\mathfrak{L})^{[t]}$ 
whose vertices lie at most
at a distance $t$ from $\mathfrak{s}
(\mathcal{O}_\mathfrak{L})$. We write $\mathfrak{s}
(\mathfrak{H})=\mathfrak{s}(\mathcal{O}_\mathfrak{L})^{[t]}$.
The vertices in $\mathfrak{s}(\mathcal{O}_\mathfrak{L})$ are
called stem vertices, while the others are called 
foliage vertices. 

For any element $\mathtt{a}\in\mathbb{M}_2(K_{\nu})$, we write
$\mathfrak{s}(\mathtt{a})$ for the largest sub-graph whose vertices
correspond to maximal orders containing the matrix $\mathtt{a}$, and we
call it the branch of $\mathtt{a}$. Such graph is 
non-empty precisely when the matrix $\mathtt{a}$ is integral 
over $\mathcal{O}_\nu$.
When the matrix
$\mathtt{a}$  generates a semisimple commutative subalgebra then
$\mathfrak{s}(\mathtt{a})=\mathfrak{s}(\mathfrak{H})$, where
$\mathfrak{H}=\mathcal{O}_\nu[\mathtt{a}]$ is as described in the
previous paragraph. In particular, when $\mathfrak{L}=K[\mathtt{a}]$ 
is the ring of diagonal matrices, then the stem of the branch 
$\mathfrak{s}(\mathtt{a})$ is the maximal path connecting $0$
and $\infty$, often called the standard apartment of 
$\mathfrak{t}$, and denoted $\mathfrak{a}_s$.

On the other hand, when the matrix $\mathtt{a}$ has a Jordan form
$\sbmattrix {\lambda}10{\lambda}$, for integral $\lambda$,
then $\mathcal{O}_\nu[\mathtt{a}]$ is a conjugate of the ring
$\mathfrak{H}_{\mathrm{nil}}=
\left\{\sbmattrix ab0a\Big|a,b\in\mathcal{O}_\nu\right\}$,
and it is called a nilpotent order. The branch
$\mathfrak{s}(\mathfrak{H}_{\mathrm{nil}})$ of the
order $\mathfrak{H}_{\mathrm{nil}}$ is the subtree whose vertices
$v_B$ are precisely those corresponding to balls $B\subseteq K$ whose
radius is $1$ or greater. The branch of any other nilpotent order
is obtained from the branch of $\mathfrak{H}_{\mathrm{nil}}$ by applying 
a suitable Moebius transformation.

We often use the notation $\mathfrak{s}(\mathfrak{H})$, where 
$\mathfrak{H}$ is an arbitrary order, for the 
intersection of the branches
of its elements, i.e., the largest sub-tree whose 
vertices correspond
to maximal orders containing $\mathfrak{H}$. Similarly, for
a representation $\phi:G\rightarrow\mathrm{GL}_2(K_\nu)$, 
we denote
by $\mathfrak{s}(\phi)$ the branch of the order spanned by
$\phi(G)$, which coincides with 
$\bigcap_{g\in G}\mathfrak{s}\big(\phi(g)\big)$,
and call it the branch of $\phi$. The set 
$\Omega$ in Proposition
\ref{percc} can be described in terms of the product 
of the vertex
sets of the branches $\mathfrak{s}_\nu(\psi_0)$ of 
a global representation $\psi_0$ at all local places
$\nu$.

\begin{Lem}\label{l31}
The neighbors of a vertex $v=v_{\mathfrak{D}}$ in the
branch $\mathfrak{s}(\phi)$ are in correspondence with
the invariant one dimensional subspaces of the image of $G$ in
the residual algebra
$\mathfrak{D}/\pi_\nu\mathfrak{D}$, which is isomorphic
to the matrix algebra $\mathbb{M}_2(\mathbb{K}_\nu)$
over the residue field $\mathbb{K}_\nu=
\mathcal{O}_\nu/\pi_\nu\mathcal{O}_\nu$.
\end{Lem}

\begin{proof}
Assume $\mathfrak{D}=\mathfrak{D}_0=\mathbb{M}_2(\mathcal{O}_\nu)$.
Then the neighbor $\mathfrak{D}'$, as in (\ref{nord}), contains
$\phi(G)$, if and only if every element of $G$ has an image
in $\mathbb{M}_2(\mathbb{K}_\nu)$ that fixes the vector $\sbvecttor10$.
Since $\mathrm{GL}_2(\mathcal{O}_\nu)$ acts transitively on both,
the neighbors of $v_0=v_{\mathfrak{D}_0}$ and, via 
$\mathrm{GL}_2(\mathbb{K}_\nu)$, on the one dimensional
subspaces of $\mathbb{K}_\nu^2$, the result follows.
\end{proof}

One more tool from the theory of local orders that we need 
in the sequel is the Artin map mentioned in the introduction.
Assume $\mathfrak{D}$ and $\mathfrak{D}'$ are two
maximal orders satisfying the relation
$\mathfrak{D}'_\nu=
\mathtt{a}_\nu\mathfrak{D}_\nu\mathtt{a}_\nu^{-1}$,
locally at every place $\nu$,
for a family $\left(\mathtt{a}_\nu\right)_\nu$
of local elements. Since $\mathfrak{D}'_\nu=\mathfrak{D}_\nu$
at almost all places, according to Property \ref{it1},
we can assume that $\left(\mathtt{a}_\nu\right)_\nu$ is an 
adele of $\mathrm{GL}_2(K)$, i.e., it satisfies
$\mathtt{a}_\nu,\mathtt{a}_\nu^{-1}\in
\mathrm{GL}_2(\mathcal{O}_\nu)$ at almost all local places
$\nu$. Then the Artin distance is defined by the formula 
$$\mathrm{d}(\mathfrak{D},\mathfrak{D}')=[J,\Sigma/K],$$
where $J$ the ideal $J$ generated by
the idele $a=\big(\det(\mathtt{a}_\nu)\big)_\nu$,
and $\Sigma$ is a suitable class field called the 
spinor class field \cite[\S1]{A13}. In this work we just need
to know that, for a matrix algebra, the spinor class field $\Sigma$ is the largest exponent
$2$ extension contained in the Hilbert class field
\cite{A03}. 
The Artin distance is trivial only between isomorphic 
maximal orders in matrix algebras. In fact, this is
the case for all central simple algebras that fail to be
completely ramified at some Archimedean place.

Two global $\nu$-variants $\mathfrak{D}$ and $\mathfrak{D}'$ 
are called $\nu$-neighbors if the completions 
$\mathfrak{D}_\nu$ and $\mathfrak{D}'_\nu$ correspond to neighboring
vertices of $\mathfrak{t}_\nu$.
The Artin distance between $\nu$-neighbors is $[P,\Sigma/K]$, where $P=P[\nu]$ is the prime ideal 
corresponding to the place $\nu$.
Since the Artin map (for unramified extensions) is trivial on principal ideals, the Artin distance between $\nu$-variants is trivial whenever
$P[\nu]$ is principal. The latter property
is essential in all that follows.

\section{Some lemmas on cyclotomic fields}

 In this section we prove some
lemmas on cyclotomic integers that
are needed in the sequel. We begin by recalling a well
known basic fact, whose proof can be found in \cite[\S I.10.2]{N99}.

\begin{Lem}\label{lc1}
    Let $\xi$ be a root of unity of order $n$. Then $\mathbb{Z}[\xi]$ 
    is the full ring of integers of the cyclotomic field 
    $\mathbb{Q}[\xi]$. \qed
\end{Lem}

\begin{ex}\label{ec1}
    It is not true in general that $\mathcal{O}_K[\xi]$
    is the full ring of integers of $K(\xi)$ for every 
    number field $K$. For example, if $K=\mathbb{Q}(\sqrt{-5})$,
    then $\mathcal{O}_K=\mathbb{Z}[\sqrt{-5}]$, but
    $\frac{1+\sqrt5}2=\frac{1+i\sqrt{-5}}2$ is an integer
    in $K(i)$.
\end{ex}

\begin{Lem}\label{lc2}
    Let $\xi$ be a root of unity of order $n$. Then
    $\xi-1$ is a prime in the ring of integers
    $\mathbb{Z}[\xi]$ of $\mathbb{Q}[\xi]$ 
    precisely when $n$ is a prime power. In any other case
    $\xi-1$ is a unit.
\end{Lem}

\begin{proof}
    If $n=p^t$ is a $p$-power, and if $\mathrm{f}_n(x)$ is 
    the $n$-th cyclotomic polynomial, then $\mathrm{f}_n(x+1)$
    is an Eisenstein polynomial at $p$, so the root $\xi-1$ is
    a uniformizer at the only place of $\mathbb{Q}[\xi]$ lying
    over $p$. At any other place, the primitive $n$-th root of
    unity cannot be congruent to $1$, since that would force
    all $n$-th roots of unity to be trivial in the 
    corresponding residue field, but this is not the case. 
    The same argument works at every place
    in case $n$ is not a prime power.
\end{proof}

\begin{Lem}\label{lc3}
    Let $\xi$ be a root of unity of order $n$. Let $m$ be an integer
    that is relatively prime to $n$. Then the ring $\mathbb{Z}[\xi]$ 
    contains a unit that is congruent to $m$ modulo $1-\xi$.
\end{Lem}

\begin{proof}
    If $\xi-1$ is a unit, there is nothing to prove, so we assume that
    $n=p^t$ is a prime power. The unit in question is
    $$\frac{\xi^m-1}{\xi-1}=1+\xi+\cdots+\xi^{m-1}\equiv m
    \ (\mathrm{mod}\ 1-\xi).$$
    It is a unit because both, the numerator and denominator of the
    fraction generate the same ideal, which is the unique maximal ideal 
    lying over $p$. The result follows.
\end{proof}

Now we consider the field 
$K=\mathbb{Q}(\xi+\xi^{-1})$, which is the largest totally real field contained in the
cyclotomic extension. It is the invariant
field of the automorphism taking $\xi$
to $\xi^{-1}$, whence its ring of integers 
is generated by the elements $\xi^k+\xi^{-k}$
for $1\leq k\leq n-1$. Since 
$\xi^k+\xi^{-k}$ is a polynomial in
$\xi+\xi^{-1}$, as follows easily from the iteration $$\xi^{k+1}+\xi^{-k-1}=
(\xi^k+\xi^{-k})(\xi+\xi^{-1})-
(\xi^{k-1}+\xi^{-k+1}),$$
we conclude that the ring of integers in $K$ is 
$\mathbb{Z}(\theta)$, where $\theta=\xi+\xi^{-1}$.

\begin{Lem}\label{lc4}
    Let $\theta=\xi+\xi^{-1}$ as above. Then
    \begin{itemize}
    \item The element $\theta-2$ is a prime 
    precisely when $n$ is a prime power. Otherwise it is a unit.
    \item The element $\theta+2$ is a prime 
    precisely when $n$ is twice a prime power. Otherwise it is a unit.
    \item when $n$ is a power of $2$, then
    $\theta-2$ and $\theta+2$ are associates, i.e.,
    they generate the same principal ideal.
    \end{itemize}
\end{Lem}

\begin{proof}
   The first statement follows since 
   $\theta-2=-(\xi-1)(\xi^{-1}-1)$. It is shown in 
   the proof of Lemma \ref{lc2} that, when
   $n$ is a power of a prime $p$, the extension
   $\mathbb{Q}(\xi)/\mathbb{Q}$ is totally
   ramified at $p$, with $\xi-1$ as a generator
   of the only ideal over $p$. Since 
   $\mathbb{Q}(\theta)/\mathbb{Q}$ is an
   intermediate extension, it is also totally
   ramified. Now the relation 
   $\theta-2=-(\xi-1)(\xi^{-1}-1)$ shows that
   $\theta-2$ is a generator of the only prime 
   ideal of $K$ containing $p$, since
   $\mathbb{Q}(\xi)/K$ is a quadratic extension,
   while $\xi^{-1}-1$ is a root of the same 
   irreducible polynomial as $\xi-1$. The second 
   statement is similarly deduced from
   $(\theta+2)=(-\xi-1)(-\xi^{-1}-1)$ if we proved
   that $-\xi$ is a root of unity of prime power 
   order precisely when $n$ is twice a prime 
   power. Now, if $n$ is the order of $\xi$,
   then the order of $-\xi$ is $2n$ whenever $n$
   is odd. This is never a prime power for odd
   $n$ unless $n$ is one. If $n$ is divisible by 
   $4$, then the order of $-\xi$ is $n$, but now
   $n$ can be a prime power only if it is
   a power of $2$. In the remaining case, 
   $n$ is twice an odd number and the order
   of $-\xi$ is $n/2$, so $n/2$ must be a prime
   power. For the final statement, when $n\geq4$ is
   a power of $2$, both $\theta-2$ and $\theta+2$
   are generators of the unique prime ideal containing $2$.
\end{proof}

\section{Decomposable representations}\label{S6}

In this section we consider a decomposable representation 
$\rho$, which we assume to be of the form 
$\rho=\rho_\chi=\sbmattrix 100{\chi}$,
as in \S\ref{MR}, since the study of all
decomposable representations can be reduced to this
case. The representation $\rho_\chi$
factors through a cyclic group, so we might assume throughout
this section that $G=\langle g_0\rangle$ is cyclic.
It suffices, therefore to study the set of maximal orders
containing $\mathtt{r}_\chi=\rho_\chi(g_0)=
\sbmattrix 100{\eta}$,
where $\eta=\chi(g_0)$.

\begin{Lem}\label{l61}
    Let $\eta \in \mathcal{O}_\nu^* $, for a place $\nu$ 
    of $K$. The action of the matrix 
    $\mathtt{r}=\begin{pmatrix} 1&0\\0&\eta \end{pmatrix}$ 
    on the Bruhat-Tits tree at $\nu$
    fixes precisely the vertices that are at distance at most 
    $\nu(1-\eta )$ of the standard apartment
    $\mathfrak{a}_s$.
\end{Lem}
\begin{proof}
    Recall that $\mathfrak{a}_s$ is the maximal
    path from $0$ to $\infty$,
    and its vertices correspond to balls containing $0$, i.e.,
    $v=v_a^{[n]}$ is in the standard apartment precisely 
    when $n\leq\nu(a)$.
    Consider a vertex $v=v_a^{[n]}$ of the 
    Bruhat-Tits tree that is 
    at a distance $m\geq0$ from the 
    standard apartment, and corresponds
    to a Ball $B_a^{[n]}$. The closest point to $v$
    in the standard apartment is $v_a^{[\nu(a)]}$, whence
    $\nu(a)=n-m$. A triplet of visual limits having $v$
    as incenter is $\left\lbrace a, a+\pi^n, 
    \infty\right\rbrace$. 
    The Moebius transformation associated with the matrix $A$ 
    sends this triplet to the set $\left\lbrace \frac{a}
    {\eta },\frac{a+\pi^n}{\eta },\infty\right\rbrace$.
    According to \ref{sic}, the vertex $v$ 
    is invariant if and only 
    if the later set also has $v$ as incenter.
      This will happen if and only if 
      the following conditions are 
      satisfied: $\frac{a}{\eta }, 
      \frac{a+\pi^n}{\eta }\in B_a^{[n]}$ 
      and $\nu\left(\frac{a}{\eta }-
      \frac{a+\pi^n}{\eta }\right)=n$. 
      The last 
      condition is clearly satisfied. For the 
      other property, note 
      that $\nu\left(a-\frac{a}{\eta }\right)=
      n-m+\nu(\eta -1)$, so we have 
      that $\frac{a}{\eta }\in B_a^{[n]}$ if and only if 
      $v(\eta -1)\geq m$. Finally, 
      $\frac{a}{\eta }\in B_a^{[n]}$ implies
      $B_{a/\eta }^{[n]}=B_a^{[n]}$, whence
      $\nu\left(\frac{a}{\eta }-
      \frac{a+\pi^n}{\eta }\right)=n$ implies
      $\frac{a+\pi^n}{\eta }\in B_a^{[n]}$. The result follows.
\end{proof}

Since the normalizer in $\mathrm{GL}_2(K_\nu)$ 
of a local maximal order
$\mathfrak{D}_\nu$ is $K_\nu^*\mathfrak{D}_\nu^*$, 
and the determinant
of the matrix $\mathtt{r}$ defined above is a unit, 
next result follows:

\begin{Cor}\label{cl61}
    Let $\eta \in \mathcal{O}_K^* $. The maximal order containing the matrix
    $\mathtt{r}=\begin{pmatrix} 1&0\\0&\eta \end{pmatrix}$,
    or the order it generates, are those corresponding to
    vertices whose distance to the standard apartment is no larger
    than $\nu(1-\eta )$.\qed
\end{Cor}

\begin{Lem}\label{l62}
    Let $T_0\subseteq\mathrm{GL}_2(K)$ be the group of diagonal matrices.
    If the order of the representation $\chi$ is not a prime power,
    then the number of $T_0$-orbits of maximal orders containing 
    $\rho_\chi(G)$ is the class number $h_K$. The number of such 
    orbits containing maximal orders isomorphic to
    $\mathbb{M}_2(\mathcal{O}_K)$ is $h_K/h_K(2)$, where
    $h_K(2)=|\mathcal{G}_K(2)|$.
\end{Lem}

\begin{proof}
    If the order of the representation $\chi$ is not a prime power,
    then the element $1-\eta $ is a unit in $\mathbb{Z}[\eta ]$, 
    and therefore also in $\mathcal{O}_K$.
    This means that the branch of the matrix $\mathtt{r}$ in
    Cor. \ref{cl61} coincides with the standard apartment,
    since $\nu(1-\eta )=0$. 
    In particular, the idempotent $\mathtt{n}=\sbmattrix1000$ 
    is contained in every local maximal order containing $\rho_\chi(G)$.
    Then every global maximal order $\mathfrak{D}$ containing 
    $\rho_\chi(G)$ satisfies $\mathtt{n}\in\mathfrak{D}_\nu$, 
    locally at every place $\nu$, and therefore also
    $\mathtt{n}\in\mathfrak{D}$. We conclude that $\mathfrak{D}=\mathfrak{D}_\Lambda$ for a lattice $\Lambda$
    that is $\mathtt{n}$-invariant, whence it follows that 
    $\Lambda=\sbvecttor{I}{I'}$ for two fractional ideals $I$ and $I'$. 
    Two such lattices, $\sbvecttor{I}{I'}$ and $\sbvecttor{J}{J'}$ 
    correspond to the same maximal order if there is a
    fractional ideal $\tilde{I}$ satisfying $J=\tilde{I}I$
    and $J'=\tilde{I}I'$. In particular $J'J^{-1}=I'I^{-1}$ is an
    invariant that completely determines the maximal order. Now 
    the result follows if we note that the group of diagonal 
    matrices act on this invariant by scalar multiplication.
    Specifically, if $\mathtt{a}=\sbmattrix a00b$, then 
    $\mathtt{a}\sbvecttor{I}{I'}$ has the invariant $ba^{-1}I'I^{-1}$.
    
    For the last statement, we need to find those invariants
    $I'I^{-1}$ for which the representative $\sbvecttor{I}{I'}$ can be 
    chosen free. The latter condition is equivalent to
    $II'$ being principal. If $II'=b\mathcal{O}_K$, then
    $I^{-1}=b^{-1}I'$, so $I'I^{-1}=b^{-1}(I')^2$.
    We conclude that the invariant corresponding to orders
    isomorphic to $\mathbb{M}_2(\mathcal{O}_K)$ are those
    that are squares in the class group. In particular, 
    the number of possible invariants  is the order
    of the image of the map $f:\mathcal{G}_K\rightarrow\mathcal{G}_K$
    defined by $f[J]=[J^2]$. This is equal
    to $h_K/h_K(2)$, since $h_K(2)$ is the order of the kernel.
\end{proof}

\begin{Prop}\label{p63}
    If the order of the representation $\chi$ is not a 
    prime power, then the number of 
    $\mathrm{GL}_2(\mathcal{O}_K)$-conjugacy classes of 
    $\mathcal{O}_K$-representations of $G$ in the
    $\mathrm{GL}_2(K)$-conjugacy class of $\rho_\chi$ is 
    the class number $h_K$. 
\end{Prop}

\begin{proof}
    According to Lemma \ref{l62}, it suffices to prove that each maximal
    order containing  $\rho_\chi(G)$ corresponds to
    $h_K(2)$ conjugacy classes of representations.
    In other words, it suffices to prove that the order of
    $\tilde{N}_a$, as defined in (\ref{defna}), is precisely
    $h_K(2)=|\tilde{N}|$, i.e., we need to prove that
    $\Theta_{\mathfrak{D}}$, as defined in Lemma \ref{Lemcc},
    is trivial. In fact, $N_{\mathfrak{D}}\cap L^*$ contains precisely the 
    diagonal matrices $\mathtt{a}=\sbmattrix a00b$ for which $ab^{-1}$
    is a unit. For such elements the ideal generated by the determinant is a square,
    whence every class in $\Theta_{\mathfrak{D}}$ is the class of principal
    ideals, and therefore is trivial. The result follows.
\end{proof}

In the preceding cases, the ring $\mathcal{O}_\nu[\mathtt{r}]$
is the full ring of integral diagonal matrices, so the branch 
$\mathfrak{s}(\mathtt{r})$, at $\nu$, coincide with
the standard apartment $\mathfrak{a}_s$. In general, the ring
$\mathcal{O}_K[\mathtt{r}]$ might be smaller at a finite
set of places $\nu$. In this section, we refer to the orders 
$\mathfrak{D}$ corresponding to vertices in $\mathfrak{a}_s$ as
the stem orders, since $\mathfrak{a}_s$ is the stem of the branch of $\mathtt{r}$. In particular, an order is a stem order at all
but finitely many places. The anchor order of any local order $\mathfrak{D}_\nu$, corresponding to a vertex $v$ is defined 
as the stem order $\mathfrak{D}'_\nu$ corresponding to the vertex
$v'$ of $\mathfrak{a}_s$ that is closest to $v$. This definition
extends naturally to global orders. The anchor order of a
global order $\mathfrak{D}$ is the global order $\mathfrak{D}'$
for which $\mathfrak{D}'_\nu$ is the anchor order of
$\mathfrak{D}_\nu$ locally at all places. We assume
$\mathfrak{D}_\nu=\mathfrak{D}'_\nu$ if $\mathfrak{D}_\nu$
is already a stem order.

\begin{Prop}\label{p64}
    Assume $K=\mathbb{Q}[\eta ]$, where $\chi$ and $\eta $
    are as above.
    If the order of the representation $\chi$ is a prime power,
    then the number of conjugacy classes of 
    $\mathcal{O}_K$-representations of $G$ in the
    $K$-conjugacy class of $\rho_\chi$ is $2h_K$. 
\end{Prop}

\begin{proof}
    In this case the element $\tilde{p}=1-\eta $ is a prime 
    lying over
    a prime $p$ of $\mathbb{Z}$, as shown in Lemma
    \ref{lc2}. In particular,
    the ideal $\mathbf{p}_1=(\tilde{p})$ is principal, whence $\nu_{\mathbf{p}_1}$-variants have trivial Artin distance. At every other place $\nu'$ of $K$,
    the maximal orders containing $\rho_\chi(G)$ are in the path
    from $0$ to $\infty$, but at $\nu_{\mathbf{p}_1}$ 
    they are either on the path or
    at distance $1$ from it, since the valuation of
    $1-\eta $ is $1$, see Fig. \ref{F4}.
    We therefore classify these maximal orders $\mathfrak{D}$ 
    into stem or leaf (non-stem) orders, according to the position of the vertex
    corresponding to
    the completion $\mathfrak{D}_{\mathbf{p}_1}$ in the Bruhat-Tits
    tree at $\nu_{\mathbf{p}_1}$. Leaf orders correspond to vertices at distance $1$ from $\mathfrak{a}_s$.
    \begin{figure}
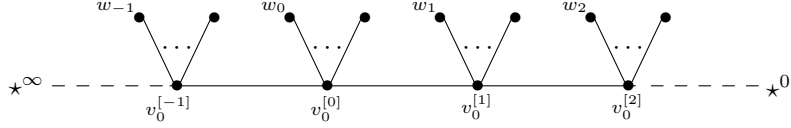

\[
\xygraph{!{<0cm,0cm>;<1cm,0cm>:<0cm,1cm>::}
!{(8,0) }*+{\star^{0}}="vl2"
!{(-2,0) }*+{\star^{\infty}}="vli"
!{(0,0) }*+{\bullet}="v0i"
!{(2,0) }*+{\bullet}="v4i"
!{(4,0) }*+{\bullet}="v5i"
!{(6,0) }*+{\bullet}="v6i"
!{(-.2,0) }*+{}="v0"
!{(6.2,0) }*+{}="v7i"
!{(-.1,-0.3) }*+{{}_{v_0^{[-1]}}}
!{(2,-0.3) }*+{{}_{v_0^{[0]}}}
!{(4,-0.3) }*+{{}^{v_0^{[1]}}}
!{(6,-0.3) }*+{{}^{v_0^{[2]}}}
!{(6.2,0) }*+{}="v7"
!{(0,-.05) }*+{}="a0" !{(-0.5,1) }*+{}="p10" !{(0.5,1) }*+{}="p20"
!{(-0.5,.9)}*+{\bullet} !{(0.5,.9) }*+{\bullet} !{(-0.8,1)}*+{{}^{w_{-1}}}
!{(0,.5)}*+{\cdots}
!{(2,-.05) }*+{}="a2" !{(1.5,1) }*+{}="p12" !{(2.5,1) }*+{}="p22"
!{(1.5,.9)}*+{\bullet} !{(2.5,.9) }*+{\bullet} !{(1.3,1)}*+{{}^{w_0}}
!{(2,.5)}*+{\cdots}
!{(4,-.05) }*+{}="a4" !{(3.5,1) }*+{}="p14" !{(4.5,1) }*+{}="p24"
!{(3.5,.9)}*+{\bullet} !{(4.5,.9) }*+{\bullet} !{(3.3,1)}*+{{}^{w_1}}
!{(4,.5)}*+{\cdots}
!{(6,-.05) }*+{}="a6" !{(5.5,1) }*+{}="p16" !{(6.5,1) }*+{}="p26"
!{(5.5,.9)}*+{\bullet} !{(6.5,.9) }*+{\bullet} !{(5.3,1)}*+{{}^{w_2}}
!{(6,.5)}*+{\cdots}
"v0"-"v7" "vli"-@{--}"v0i" "vl2"-@{--}"v7i"  
"a0"-"p10" "a0"-"p20" "a2"-"p12" "a2"-"p22"
"a4"-"p14" "a4"-"p24" "a6"-"p16" "a6"-"p26"
}
\]
\caption{Local vertices containing the representation $\rho_\chi$, 
when the order of $\chi$ is a prime power. The sample leaves are 
$w_i=v_{p_1^i}^{[i+1]}$.}\label{F4}
\end{figure}
    Stem orders belong to $h_K/h_K(2)$ different $H$-orbits, each
    corresponding to $h_K(2)$ homothety classes of lattices, 
    precisely as in the proof of the preceding proposition.
    We concentrate, therefore, on the leaf orders. 

    Diagonal matrices whose diagonal coefficients generate the same
    ideal leave every stem order invariant,  while permuting the
    leaf orders with the same anchor order. More precisely, a diagonal 
    matrix $\sbmattrix 100a$ acts on the Bruhat-Tits tree as the Moebius 
    transformation $\tau(z)=a^{-1}z$. This defines an action of
    units on leaf orders that we use in the sequel. Every
    vertex in $\mathfrak{a}_s$ has $p-1$ neighbors outside 
    $\mathfrak{a}_s$, since $K_{\tilde{p}}/\mathbb{Q}_p$ is totally
    ramified. The group $(\mathbb{Z}/p\mathbb{Z})^*$ acts on
    these neighbors without fix points, and therefore transitively,
    because of Lemma \ref{l61}.  Lemma \ref{lc3} shows now
    that global units act transitively on the leaf orders
    with a fix anchor order. Furthermore, a leaf order
    is conjugate to $\mathbb{M}_2(\mathcal{O}_K)$ if and
    only if its anchor order is, since they are $\nu_{\mathbf{p}_1}$-variants, whence their
    Artin distance is trivial. It suffices now to 
    compute the order of $\tilde{N}_a$ for one of these 
    leaf orders. In fact, every diagonal element leaving a
    leaf order invariant must also leave invariant the 
    corresponding anchor order. The proof that 
    $\Theta_{\mathfrak{D}}$ is trivial goes, therefore, as before.
    The result follows.   
\end{proof}

To generalize the previous result to an arbitrary extension,
we need some additional notation. Let $K$ be an arbitrary field containing
$\mathbb{Q}[\eta ]$. Let $\tilde{p}\mathcal{O}_K=
P_1^{e_1}\cdots P_r^{e_r}$ be the prime decomposition of
the ideal  $\tilde{p}\mathcal{O}_K\subseteq\mathcal{O}_K$. Then,
it follows from the identity $\tilde{p}=1-\eta $ and
Cor. \ref{cl61} that the maximal orders $\mathfrak{D}$ containing 
$\mathtt{r}=\mathtt{r}_\chi$ 
are those satisfying both of the following conditions:
\begin{itemize}
\item The vertex $v_{\mathfrak{D},Q}$, corresponding to $\mathfrak{D}_Q$, 
belongs to $\mathfrak{a}_s$ at every place $Q\notin\{P_1,\dots,P_r\}$,
\item the distance from the vertex $v_{\mathfrak{D},P_i}$ to
$\mathfrak{a}_s$ is at most $e_i$, for every $i\in\{1,\dots,r\}$.
\end{itemize}
Locally, the set of vertices corresponding to maximal orders containing $\mathtt{r}$ 
with the same anchor order is a side branch, as described at the
end of the introduction.
The side branch is non-trivial only for the places $\nu_{P_1},\dots,\nu_{P_r}$. At each of
these places $\nu_{P_i}$, the side 
branch contains precisely the vertices at a distance $e_i$ or less from
the stem. Globally, therefore,
the set of orders with a given anchor order is
in correspondence with the vertices in the product of the side
branches corresponding to the places $P_1,\dots,P_r$.
The group of units $\mathcal{O}_K^*$
acts via Moebius transformation as described before,
leaving invariant every ball containing $0$,
i.e., every vertex on the standard apartment, and
therefore acts on every side branch. This defines a
coordinate-wise action on the product of side branches. 
Locally, this action obeys Lemma
\ref{l61}, which is used in explicit examples below to
count orbit numbers. 

\begin{Lem}\label{lsd1}
    Side branches corresponding to different anchor orders can be 
    identified, non canonically, in a way that commutes with the
    action of $\mathcal{O}_K^*$.
\end{Lem}

\begin{proof}
In the local Bruhat-Tits we can permute
transitively the stem orders, i.e., the vertices in the standard 
apartment, multiplying by powers of
a local uniformizer $\pi_\nu$. This action commutes with the multiplication by units
described above. This can be done at simultaneously at all places, but  depends on the choice of the uniformizers $\pi_\nu$, whence it is non-canonical.
\end{proof}

\begin{proof}[Proof of Theorem \ref{t5}]
The local multiplication by
uniformizers described above
can be interpreted as a transitive
free action of the ideal group
$\mathcal{I}_K$ on the set
of stem orders. 
As before, the group of diagonal matrices acts on the set
of stem orders with $h_K/h_K(2)$ orbits, and every
maximal order corresponds to $h_K(2)$ homothety classes
of lattices. 

Since the prime ideals $P_i$ no longer need to
be principal, it is not necessarily true that an order
$\mathfrak{D}$ is isomorphic to $\mathbb{M}_2(\mathcal{O}_K)$
if so is its anchor order. However, the Artin distance 
 between two maximal orders is trivial
whenever the local graph distance between those orders
is even at all places. 
In other words, the isomorphism class of a maximal order
depends only on the set
of places at which the corresponding vertex lie at an
odd distance from a fixed vertex,
say the vertex $v_0^{[0]}$
corresponding to $\mathbb{M}_2(\mathcal{O}_K)$.
This mean that, whenever an order 
$\mathfrak{D}$ is in the wrong isomorphism class, 
this can be fixed by shifting the corresponding anchor vertex 
by one at a finite set of places, namely those places $\nu$ at
which the distance from $v_{\mathfrak{D},\nu}$ to $v_0^{[0]}$ 
is odd.
This is ilustrated in Fig. \ref{f3p}. If the neighbor of
$v_0$ outside the stem corresponds
to a maximal order in the wrong
isomorphism class, but $v_0$ does not, then we fix it by 
replacing the neighbor by the vertex $w$, which is at the 
same relative position with respect to a 
neighboring anchor order.

Fix an identification between side branches with different
anchor order, as in Lemma \ref{lsd1}.
Let $o$ be a orbit in the side branch.
Using the procedure described above, we prove that 
the number of orbits, under conjugation by the group 
$T_0$ of diagonal matrices, of maximal orders 
$\mathfrak{D}\cong\mathbb{M}_2(\mathcal{O}_K)$
that correspond to vertices in $o$, is the same 
as the number of orbits of anchor orders in the 
corresponding isomorphism class.
Reasoning as in Lemma \ref{l62}, we see that this 
number is $h_K/h_K(2)$, since the proof of that 
lemma works for stem orders
at an arbitrary field, and the orders in a 
different isomorphism class correspond to a coset. 
Now we prove the triviality of the
set $\Theta_{\mathfrak{D}}$, for every order $\mathfrak{D}$,
exactly as in Prop. \ref{p64}. The result follows.
\end{proof}

\begin{ex}\label{e76n}
Fig. \ref{f2p} shows some examples of products of side
branches, to illustrate some
possible behaviors.
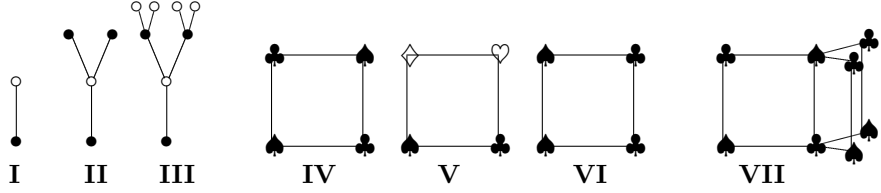
\begin{figure}
\unitlength 1mm 
\linethickness{0.4pt}
\ifx\plotpoint\undefined\newsavebox{\plotpoint}\fi 
\begin{picture}(93.5,23.75)(-20,24)
\put(9,25){$\mathbf{I}$}\put(19,25){$\mathbf{II}$}
\put(29,25){$\mathbf{III}$}\put(48,25){$\mathbf{IV}$}
\put(66,25){$\mathbf{V}$}\put(84,25){$\mathbf{VI}$}
\put(106,25){$\mathbf{VII}$}
\put(10,30){\line(0,1){8}}
\put(9.1,29.7){$\bullet$}
\put(9.1,37.7){$\circ$}
\put(20,30){\line(0,1){8}}
\multiput(19.7,39)(-.03333333,.08){75}{\line(0,1){.08}}
\multiput(20.3,39)(.03333333,.08){75}{\line(0,1){.08}}
\put(19.1,29.7){$\bullet$}
\put(19.1,37.7){$\circ$}
\put(16.1,44){$\bullet$}\put(21.9,44){$\bullet$}
\put(30,30){\line(0,1){8}}
\multiput(29.7,39)(-.03333333,.08){75}{\line(0,1){.08}}
\multiput(30.3,39)(.03333333,.08){75}{\line(0,1){.08}}
\put(29.1,29.7){$\bullet$}
\put(29.1,37.7){$\circ$}
\put(26.3,44){$\bullet$}\put(31.9,44){$\bullet$}
\multiput(27.3,44.3)(.03333333,.12){30}{\line(0,1){.08}}
\multiput(27.3,44.3)(-.03333333,.12){30}{\line(0,1){.08}}
\multiput(32.7,44.3)(.03333333,.12){30}{\line(0,1){.08}}
\multiput(32.7,44.3)(-.03333333,.12){30}{\line(0,1){.08}}
\put(25.1,47.7){$\circ$}\put(27.5,47.7){$\circ$}
\put(30.5,47.7){$\circ$}\put(32.9,47.7){$\circ$}
\put(44,30){\framebox(12,12)[]{}}
\put(43,29){$\spadesuit$}\put(55,29){$\clubsuit$}
\put(43,41){$\clubsuit$}\put(55,41){$\spadesuit$}
\put(62,30){\framebox(12,12)[]{}}
\put(61,29){$\spadesuit$}\put(73,29){$\clubsuit$}
\put(61,41){$\diamondsuit$}\put(73,41){$\heartsuit$}
\put(80,30){\framebox(12,12)[]{}}
\put(79,29){$\spadesuit$}\put(91,29){$\clubsuit$}
\put(79,41){$\spadesuit$}\put(91,41){$\clubsuit$}
\put(104,30){\framebox(12,12)[cc]{}}
\multiput(116,30)(.2282609,-.0326087){23}{\line(1,0){.2282609}}
\multiput(116,42)(.2282609,-.0326087){23}{\line(1,0){.2282609}}
\multiput(116,30)(.12980769,.03365385){52}{\line(1,0){.12980769}}
\multiput(116,42)(.12980769,.03365385){52}{\line(1,0){.12980769}}
\put(122.4,31.8){\line(0,1){12}}
\put(121,29){\line(0,1){12}}
\put(103,29){$\spadesuit$}\put(115,29){$\clubsuit$}
\put(103,41){$\clubsuit$}\put(115,41){$\spadesuit$}
\put(120,28){$\spadesuit$}\put(122,31){$\spadesuit$}
\put(120,40){$\clubsuit$}\put(122,43){$\clubsuit$}
\end{picture}
\caption{Some products of side branches. }\label{f2p}
\end{figure}
In sub-figures
$\mathbf{I}$-$\mathbf{III}$ in Figure
\ref{f2p}, we assume $r=1$, and we draw the part of the local 
branch at $P_1$ containing all vertices corresponding to local 
orders with the same anchor order. The ramification index
$e_1$ is $1$ in $\mathbf{I}$,
$2$ in $\mathbf{II}$ and $3$ in $\mathbf{III}$. 
The vertices marked with
bullets are at an even distance from the anchor vertex
at the base, while white circles denote vertices at an odd
distance.
The remaining products in Figure \ref{f2p} correspond to cases 
where $r=2$. In sub-figures $\mathbf{IV}$-$\mathbf{VI}$
we have $e_1=e_2=1$,
while sub-figure $\mathbf{VII}$ shows a case where 
$e_1=2$ and $e_2=1$.
In sub-figure $\mathbf{IV}$, the class distance between both 
$P_1$-neighbors and
$P_2$ neighbors is the same, whence the anchor order 
(corresponding to the lower left corner) 
belong to the same class
as the order corresponding to the upper right corner, which is 
locally at distance one at both places. To put the order 
corresponding to any of the remaining two corners in the right
class we need to replace the anchor vertex by a neighbor in the
stem, at either of the places $P_1$ or $P_2$.
In sub-figure $\mathbf{V}$ the class distance 
between $P_1$-neighbors is different from the distance between
$P_2$ neighbors. This means that the corresponding 
orders belong 
to four different classes and we need four different shifts to
take all cases into account. In sub-figure $\mathbf{VI}$ 
one of the two places has a
trivial image under the corresponding Artin map, whence shifts
on that direction have no effect on the isomorphism 
class of the order.
\end{ex}

\begin{rem}
    Since $\tilde{p}\mathcal{O}_K=P_1^{e_1}\cdots P_r^{e_r}$ is
    a principal ideal, its image under the 
    Artin map is trivial,
    so pictures $\mathbf{I}$, $\mathbf{III}$ and 
    $\mathbf{V}$-$\mathbf{VII}$ are not realistic. 
    But they can 
    still depict factors of larger products. 
\end{rem}

\begin{ex}\label{e67}
    Consider the case $n=2$, i.e., we study the representation
    $\rho_\chi$, where $\chi$ is the only non-trivial 
    representation of the cyclic group $C_2$ with two
    elements, which sends the generator $g$ to the diagonal
    matrix $\rho_{\chi}(g)=\sbmattrix100{-1}$. We compute 
    the integral representations for the field
   $K=\mathbb{Q}(\sqrt{-5} )$. In this case, we have
   $2\mathcal{O}_K=\mathbf{2}^{2}_1$, so $r=1$ and $e_1=2$.
   In other words, $\mathbb[B]$ is a side branch looking
   like the one depicted in sub-figure $\mathbf{II}$ of
   Figure \ref{f2p}. 
   Note that $\mathcal{O}_K^*=\{\pm1\}$ acts trivially
   on the side branch, since $1-(-1)=2$ generates 
   $\mathbf{2}^{2}_1$, and we can apply Lemma \ref{l61}.
   We conclude that there are $4h_K=8$ conjugacy classes
   of integral representations.
   Since $\mathbf{2}_1=(2,1+\sqrt{-5})$ belong to the unique
   non-trivial ideal class, $\mathbf{2}_1$-neighboring orders
   belong to two different isomorphism classes, so the relevant
   vertices in the Bruhat-Tits tree at $2$ can be chosen as in 
   Figure \ref{f3p}, where $v_0=v_0^{[0]}$ corresponds 
   to the ball 
   $B_0^{[0]}$, and therefore to the order $\mathfrak{D}_0=\mathbb{M}_2(\mathcal{O}_K)$.

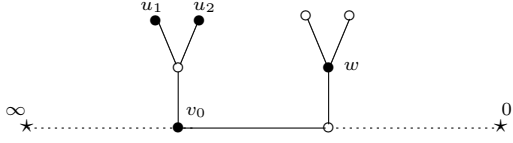
\begin{figure}
\unitlength 1mm 
\linethickness{0.4pt}
\ifx\plotpoint\undefined\newsavebox{\plotpoint}\fi 
\begin{picture}(123.5,23.75)(-50,30)
\put(20,30){\line(0,1){8}}
\put(19.6,30.5){\line(1,0){20}}
\multiput(0,30.5)(0.9,0){25}{\line(1,0){.3}}
\multiput(41,30.5)(0.9,0){25}{\line(1,0){.3}}
\multiput(19.7,39)(-.03333333,.08){75}{\line(0,1){.08}}
\multiput(20.3,39)(.03333333,.08){75}{\line(0,1){.08}}
\put(19.1,29.7){$\bullet$}
\put(19.1,37.7){$\circ$}
\put(16.1,44){$\bullet$}\put(21.9,44){$\bullet$}
\put(40,31){\line(0,1){8}}
\multiput(39.7,39)(-.03333333,.08){75}{\line(0,1){.08}}
\multiput(40.3,39)(.03333333,.08){75}{\line(0,1){.08}}
\put(39.1,29.7){$\circ$}
\put(39.1,37.7){$\bullet$}
\put(36.1,44.6){$\circ$}\put(41.9,44.6){$\circ$}
\put(-1,30){$\star$} \put(62,30){$\star$}
\put(-3,31){${}^\infty$} \put(63,31){${}^0$}
\put(21,31){${}^{v_0}$} \put(42,37){${}^{w}$}
\put(15,45){${}^{u_1}$} \put(22,45){${}^{u_2}$}
\end{picture}
\caption{Relevant vertices at $\mathbf{2}_1$ for Example
\ref{e67}. }\label{f3p}
\end{figure}
   It follows by inspection that $w$ corresponds to the ball
   $B_\pi^{[2]}$, while the vertices $u_1$ and $u_2$
   correspond to the balls $B_1^{[2]}$ and $B_{1+\pi}^{[2]}$.
   Note that, if $B=B_a^{[n]}$, a local matrix 
   $\mathtt{t}$ satisfying
   $\mathtt{t}\mathfrak{D}_0\mathtt{t}^{-1}=\mathfrak{D}_B$ is 
   $\mathtt{t}=\sbmattrix ah10$, for any element $h$ whose 
   valuation is $n$. One of the representations defined
   by the corresponding vertex $x$ is the conjugate
   $\varphi_x=\mathtt{t}^{-1}\rho_\chi \mathtt{t}$, 
   while the other is $\lambda_x=
   \mathtt{u}^{-1}\mathtt{t}^{-1}\rho_\chi 
   \mathtt{t} \mathtt{u}$, where 
   $\mathtt{u}=
   \sbmattrix{3+\sqrt{-5}}8{4+2\sqrt{-5}}{13+\sqrt{-5}}$ is
   a matrix whose columns are generators of the module
   $(\mathbf{2}_1,\mathbf{2}_1)$, since this matrix 
   belongs to the normalizer of $\mathfrak{D}_0$.
   A full list of representatives of all 
   integral representations
   can be seen in Table \ref{T1}.
   \begin{table}
\begin{center}
\begin{tabular}{||c c c||} 
 \hline
 $x$ & $\varphi_x(g)$ & $\lambda_x(g)$  \\
 \hline\hline
 $v_0$ & 
 $\left[ \begin{array}{cc} 1 & 0 \\ 0 & -1 
 \end{array}\right]^{\phantom{2}}_{\phantom{2}}$ 
 & $\left[ \begin{array}{cc} 33+16\sqrt{-5} & 
 104+8\sqrt{-5}\\ -2-10\sqrt{-5} & -33-16\sqrt{-5}
 \end{array}\right]^{\phantom{2}}_{\phantom{2}}$   \\ 
 \hline
 $w$ & 
 $\left[ \begin{array}{cc} -1 & 0 \\ 1+\sqrt{-5} & 1
 \end{array}\right]^{\phantom{2}}_{\phantom{2}}$ 
 & $\left[ \begin{array}{cc} -25-32\sqrt{-5} & 
 -136-40\sqrt{-5}\\ -11+15\sqrt{-5} & 25+32\sqrt{-5}
 \end{array}\right]^{\phantom{2}}_{\phantom{2}}$   \\ 
 \hline
  $u_1$ & 
 $\left[ \begin{array}{cc} -1 & 0 \\ 1 & 1 
 \end{array}\right]^{\phantom{2}}_{\phantom{2}}$ 
 & $\left[ \begin{array}{cc} -45-20\sqrt{-5} & 
 -136-8\sqrt{-5}\\ 4+13\sqrt{-5} & 45+20\sqrt{-5} 
 \end{array}\right]^{\phantom{2}}_{\phantom{2}}$   \\ 
 \hline
  $u_2$ & 
 $\left[ \begin{array}{cc} -1 & 0 \\ 2+\sqrt{-5} & 1 
 \end{array}\right]^{\phantom{2}}_{\phantom{2}}$ 
 & $\left[ \begin{array}{cc} -37-36\sqrt{-5} & 
 -168-40\sqrt{-5}\\ -9+18\sqrt{-5} & 37+36\sqrt{-5} 
 \end{array}\right]^{\phantom{2}}_{\phantom{2}}$   \\ 
 \hline
\end{tabular}
\end{center}
\caption{Representatives of all conjugacy classes of
non-diagonal integral representations of $C_2$ over 
the field $\mathbb{Q}[\sqrt{-5}]$.}\label{T1}
\end{table}
\end{ex}

\section{Branches and eigenvalues}

Before we go on to generalize the results in the previous
section to indecomposable abelian representation,
we need a way to compute
the part of a branch that is fixed under 
the action by conjugation
of the group $\mathfrak{L}^*$ of invertible elements in a two 
dimensional algebra $\mathfrak{L}=K[\mathtt{r}]\subseteq
\mathbb{M}_2(K)$. These results can be found in \cite{AB19},
but we include them here for the sake of completeness.

\begin{Lem}\label{l61b}
     Let $\mathtt{r}$ be a matrix with two 
     distinct eigenvalues $a$ and $b$ in $K_\nu$,
     whose valuations $\nu(a)$ and $\nu(b)$ coincide. 
    Then conjugation by $\mathtt{r}$
    fixes precisely the vertices that are at distance at most 
    $\nu\left(1-\frac ab\right)$ of the maximal path whose 
    visual limits are the fixed points of the Moebius 
    transformation defined by the matrix $\mathtt{r}$.
\end{Lem}

\begin{proof}
    This is a straightforward 
generalization of Lemma \ref{l61} if we observe that
the action by conjugation of $\mathrm{GL}_2(K_\nu)$ 
is transitive on the set of
algebras isomorphic to $K_\nu\times K_\nu$.
The statement about the visual limits is immediate if
we observe that $\infty$ and $0$ are both, the visual
limits of the standard apartment, and the fixed points
of the Moebius transformation $\mu(z)=\frac{az}{b}$
defined by the diagonal matrix $\sbmattrix{a}{0}{0}{b}$.
\end{proof}

In the following two results we consider the case where
the matrix $\mathtt{r}$ has eigenvalues in a quadratic
extension $L_\omega$ of $K_\nu$. These results can be 
obtained from Lemma \ref{l61b} by considering the
Bruhat-Tits tree $\mathfrak{t}'$ corresponding
to this field. Its vertices correspond to balls
in $L_\omega$. If the valuation $\omega$ is an 
extension of the valuation $\nu$, every ball
$B_a^{[n]}\subseteq K_\nu$, as in \ref{ballenc}, 
defines a unique
ball in $L_\omega$ with the same radius and center,
which we call its extension to $L_\omega$.
When $L_\omega/K_\nu$ is unramified, the extensions
of balls corresponding to neighboring vertices
in $\mathfrak{t}$ correspond to neighbors
in $\mathfrak{t}'$. This defines
an embedding of graphs from $\mathfrak{t}$ into 
$\mathfrak{t}'$. This is no longer true for ramified
extensions. In fact, if $e$ is the ramification index
of an extension $L_\omega/K_\nu$, then the extension
of balls corresponding to neighbors in $\mathfrak{t}$
correspond to vertices at distance $e$ in the tree
$\mathfrak{t}'$. We can fix
this issue by defining a graph $\mathfrak{t}(e)$,
which can be obtained from $\mathfrak{t}$ by replacing
every edge by a path of length $e$, as shown in 
Fig. \ref{fnew1}. The graph $\mathfrak{t}(e)$ embeds
into $\mathfrak{t}'$ for any extension $L_\omega/K_\nu$
of local fields. The graph $\mathfrak{t}_b=\mathfrak{t}(2)$
is called the barycentric subdivision. 
\begin{figure}
    \unitlength 1mm 
\linethickness{0.4pt}
\ifx\plotpoint\undefined\newsavebox{\plotpoint}\fi 
    \[
    \begin{picture}(30,14)(0,40)
\put(15.2,43.2){$\bullet$}
\put(15.2,55.4){$\bullet$}
\put(16,44){\line(0,1){12}}
\put(3.2,43.2){$\bullet$}
\put(3.2,55.4){$\bullet$}
\put(4,44){\line(0,1){12}}
\put(4,44){\line(1,0){24}}
\put(27.2,43.2){$\bullet$}
\end{picture}\qquad\qquad 
\begin{picture}(30,14)(0,40)
\put(15.2,43.2){$\bullet$}
\put(15.2,55.4){$\bullet$}
\put(16,44){\line(0,1){12}}
\put(15.2,47.2){$\bullet$}
\put(15.2,51.4){$\bullet$}
\put(3.2,43.2){$\bullet$}
\put(3.2,55.4){$\bullet$}
\put(4,44){\line(0,1){12}}
\put(3.2,47.2){$\bullet$}
\put(3.2,51.4){$\bullet$}
\put(4,44){\line(1,0){24}}
\put(7.2,43.2){$\bullet$}
\put(11.2,43.4){$\bullet$}
\put(19.2,43.2){$\bullet$}
\put(23.2,43.4){$\bullet$}
\put(27.2,43.2){$\bullet$}
\end{picture}
\]
    \caption{A graph $\mathfrak{g}$ (left) and its 
    subdivision $\mathfrak{g}(3)$ (right).}
    \label{fnew1}
\end{figure}
We use $\pi_\omega$
for a uniformizing parameter of $L_\omega$.

\begin{Lem}\label{l61c}
     Let $\mathtt{r}$ be a matrix whose eigenvalues $a,b$
     lie in an unramified quadratic extension 
     $L_\omega$ of $K_\nu$,
     and let us denote by $\omega$ the extension 
     of the valuation
     $\nu$ to this field. Then conjugation by $\mathtt{r}$
    fixes precisely the vertices that are at distance at most 
    $\omega\left(1-\frac ab\right)$ from the vertex $v$ 
    in $\mathfrak{t}$ corresponding
    to the unique maximal order in
    $\mathbb{M}_2(K_\nu)$ containing the ring 
    of integers in 
    $\mathfrak{L}_\nu=K_\nu[\mathtt{r}]\cong L_\omega$. 
    The vertex $v$ corresponds to
    a vertex in the stem $\mathfrak{p}$ of the branch
    $\mathfrak{s}_\omega(\mathtt{r})\subseteq \mathfrak{t}'$.
\end{Lem}

\begin{proof}
     We apply Lemma \ref{l61b}. The maximal 
     $\mathcal{O}_{L_\omega}$-orders
     in $\mathbb{M}_2(L_\omega)$ fixed by conjugation by
     $\mathtt{r}$ are those at a distance 
     $\omega\left(1-\frac ab\right)$ or less from 
     $\mathfrak{p}$. The visual limits of
     the path $\mathfrak{p}$ are not in 
     $K_\nu$, as they correspond to the fixed points
     in the Moebius transformation, or, equivalently, the 
     eigenvectors of $\mathtt{r}$. In particular, the 
     generator $\sigma$ of the Galois group 
    $\mathrm{Gal}(L_\omega/K_\nu)$ interchange the visual 
    limits of this path, whence $\mathfrak{p}$ has a unique 
    Galois-invariant point. This implies that the path
    $\mathfrak{p}$ and the tree $\mathfrak{t}$, which is 
    pointwise Galois-invariant, cannot have more than one 
    point in common. The only Galois invariant point
    in $\mathfrak{p}$ is, therefore, the vertex of the 
    path whose distance from $\mathfrak{t}$ is minimal.
    Since Moebius transformations of the form $\mu(z)=az+b$
    with $a\in K_\nu^*$ and $b\in K_\nu$ act transitively on
    $L_\omega\smallsetminus K_\nu$, we can assume that the
    visual limits of the path are units whose image
    in the residue field generate the extension of 
    residue fields. Since the visual limits are a Galois
    orbit, they have different images on the residue field, 
    whence the maximal path connecting them passes through
    the vertex $v=v_0^{[0]}$ corresponding to the ball
    $B_0^{[0]}=
    \{x\in L_\omega|\omega(x)\geq0\}=\mathcal{O}_{L_\omega}$, 
    which is a vertex in $\mathfrak{t}$.
    We conclude that the distance from a vertex
    in $\mathfrak{t}$ to the path is the same as its distance
    to the unique vertex $v$ of this path that is contained
    in $\mathfrak{t}$. Since the ring 
    of integers in 
    $\mathfrak{L}_\nu=K_\nu[\mathtt{r}]$ is contained
    in a unique maximal order $\mathfrak{D}'$ in 
    $\mathbb{M}_2(K_\nu)$, but
    it must be contained in the maximal order 
    corresponding to every 
    vertex in the path $\mathfrak{p}$, 
    then $\mathfrak{D}'$ corresponds to $v$.
    The result follows. 
\end{proof}

\begin{Lem}\label{l61d}
     Let $\mathtt{r}$ be a matrix whose eigenvalues $a,b$
     lie in a ramified quadratic extension 
     $L_\omega$ of $K_\nu$. 
       Let $\omega$ be the (integral valued) 
       extension of $2\nu$
     to the field $L_\omega$. We identify
     $L_\omega$ with the algebra
     $\mathfrak{L}_\nu=K_\nu[\mathtt{r}]$. 
     If $\omega(\mathtt{r})$ is odd,
    then conjugation by $\mathtt{r}$ fixes no vertex
    in the Bruhat-Tits tree $\mathfrak{t}$.
    Furthermore, there is a constant $\kappa$, such that,
    if $\omega(\mathtt{r})$ is even,
    then conjugation by $\mathtt{r}$ 
    fixes precisely the vertices that are at distance at most 
    $d=\frac12\big(\omega
    \left(1-\frac ab\right)-\kappa-1\big)$ 
    from one the vertices 
    corresponding to the two maximal orders 
    containing the ring 
    of integers in 
    $\mathfrak{L}_\nu=K_\nu[\mathtt{r}]$. Furthermore,
    such $d$ is always a positive integer when 
    $\omega(\mathtt{r})$ is even.
\end{Lem}

\begin{proof}
    In this case, the graph that naturally embeds into 
    $\mathfrak{t}'$ is the barycentric subdivision 
    $\mathfrak{t}_b$ of $\mathfrak{t}$, as discussed
    before Lemma \ref{l61c}. It is proved
    as before that the Galois action has a unique invariant
    point in the path $\mathfrak{p}$ connecting the 
    invariant points (visual limits) of
    the Moebius transformation corresponding to
    $\mathtt{r}$, and such invariant
    point is the vertex $v$ that is closest
    to $\mathfrak{t}_b$. In this case, by applying
    a transformation of the form $\mu(z)=az+b$, 
    we can assume that the 
    visual limits of the path are uniformizing parameters 
    of the quadratic extensions, whence the vertex $w$ 
    in $\mathfrak{t}_b$
    that is closest to the path is the barycenter
    of an edge in $\mathfrak{t}$, which, with the 
    assumption on the visual limits, would be the
    edge joining $v_0^{[0]}$ to $v_0^{[1]}$ in
    $\mathfrak{t}$.
    Let $\kappa$ be the distance from $v$ to $w$.

    Let $\mathfrak{H}$ be the ring of integers in 
    $\mathfrak{L}_\nu=K_\nu[\mathtt{r}]
    \subseteq\mathbb{M}_2(K_\nu)$.
    Since $\mathcal{O}_{L_\omega}\mathfrak{H}$ is an order 
    in $L_\omega[\mathtt{r}]$,
    then the maximal orders in 
    $\mathbb{M}_2(L_\omega)$ containing
    $\mathfrak{H}$ correspond to a tubular neighborhood of the 
    path. Therefore, the two maximal orders in 
    $\mathbb{M}_2(K_\nu)$ containing it
    must corresponds to the endpoints of the edge in 
    $\mathfrak{t}$ having $w$ as barycenter. It follows
    that a vertex, in $\mathfrak{t}$, is at distance 
     $\omega\left(1-\frac ab\right)$ from the path,
     if and only if it lies at a distance 
     $\omega\left(1-\frac ab\right)-\kappa$ from $w$, or 
     equivalently,  at a distance 
     $\omega\left(1-\frac ab\right)-\kappa-1$ from one
     of the endpoints of the edge in 
    $\mathfrak{t}$ having a $w$ as barycenter. 
    This is a distance
    in  $\mathfrak{t}'$, so to obtain the corresponding
    distance in  $\mathfrak{t}$ we must divide by two.
    
    Now, fix a uniformizer $\pi_\omega\in L_\omega$, and
    let $\lambda=x+y\pi_\omega\in L_\omega$, 
    with $x,y\in K_\nu$ 
    be arbitrary. Note that $x$ and  
    $y\pi_\omega$ have valuations
    with different parity. If $y\pi_\omega$
    is dominant, then $x/\lambda$
    has positive valuation, whence the relation
    $$\frac{\sigma(\pi_\omega)}{\pi_\omega}-
    \frac{\sigma(\lambda)}{\lambda}=\frac x{\lambda}
    \left(\frac{\sigma(\pi_\omega)}{\pi_\omega}-1\right),$$
    implies that $\frac{\sigma(\lambda)}{\lambda}-1$ 
    has the same 
    valuation as $\frac{\sigma(\pi_\omega)}{\pi_\omega}-1$.
    On the other hand, if $x$ is dominant, then
    $$\frac{\sigma(\lambda)}{\lambda}-1=
    \frac {y\pi_\omega}{\lambda}
    \left(\frac{\sigma(\pi_\omega)}{\pi_\omega}-1\right),$$
    has a larger valuation than 
    $\frac{\sigma(\pi_\omega)}{\pi_\omega}-1$.

    Finally, if $\mathtt{r}$ is a uniformizer 
    in $\mathfrak{L}_\nu$,
    its determinant, as a matrix, is a uniformizer of $K_\nu$, 
    and therefore the Moebius transformation associated to 
    $\mathtt{r}$ cannot leave invariant any vertex of 
    $\mathfrak{t}$. It must, however, leave the path, and 
    therefore the vertex $w$, invariant. It follows that
    $$\omega\Big(\frac{\sigma(\pi_\omega)}{\pi_\omega}-1\Big)=
    \kappa.$$
    Now we conclude from the preceding computations that
    invariant vertices in $\mathfrak{t}$ exists precisely when
    $\omega(\mathtt{r})$ is even, since this coincide with the 
    condition that $\omega\left(1-\frac ab\right)-\kappa-1$
    is non-negative.
\end{proof}

Now let $L/K$ be an extension of global fields.
We define an ideal $I=I[L/K]\subseteq\mathcal{O}_L$ 
via local conditions as follows:
If $L/K$ is unramified (split or inert) at a place 
$\omega$ of $L$,
we define $I_\omega=\mathcal{O}_{L,\omega}$. 
If $L/K$ is ramified 
we define $$I_\omega=(\pi_\omega)^{\kappa+1}=
\pi_\omega\left(\frac{\sigma(\pi_\omega)}{\pi_\omega}-1\right)
=\big(\sigma(\pi_\omega)-\pi_\omega\big).$$ 
This is indeed the different of the quadratic extension. See
for instance Prop. (2.4) in \cite[\S III.2]{N99}.
Now we are ready to prove the main result of this section:

\begin{Lem}\label{l74}
    Let $\mathtt{r}\in\mathbb{M}_2(K)$ be a matrix whose 
    eigenvalues $a,b$ lie in a quadratic extension $L$ of $K$.
    Then there is a maximal order that is invariant under 
    conjugation by $\mathtt{r}$ if and only if $\frac ab$ 
    is a unit and
    $\frac ab\equiv 1\mod I$, where 
    $I=I[L/K]\subseteq\mathcal{O}_L$ is the ideal 
    defined above.
\end{Lem}

\begin{proof}
    The fact that $\frac ab$ is a unit 
    guarantee that conjugation by $\mathtt{r}$
    fixes precisely the vertices in a path, locally 
    at all places splitting $L/K$, where 
    $L\cong K[\mathtt{r}]$.  At inert places, 
    conjugation by $\mathtt{r}$
    fixes the vertex corresponding to the unique 
    maximal order containing
    the maximal order of $K[\mathtt{r}]$, 
    since it leaves the commutative 
    subalgebra invariant. Finally, at 
    ramified places, the condition
    follows from the congruence, considering the proof of 
    the preceding lemma.
    We conclude that the given condition
    on $\frac ab$ is equivalent to $\mathtt{r}$ having at least
    one invariant vertex in every local Bruhat-Tits tree.

    Since conjugation by $\mathtt{r}$ leaves the order 
    $\mathbb{M}_2(\mathcal{O}_K)$ 
    invariant at every place at which 
    $\mathtt{r}\in\mathbb{M}_2(\mathcal{O}_K)^*$,
    i.e., at every place outside a finite set, the result is 
    a consequence of the
    correspondence between global lattices 
    and coherent families 
    described by properties 
    \ref{it1}-\ref{it3} in \S2.
\end{proof}

\begin{Cor}\label{c741}
    Let $\mathtt{r}\in\mathbb{M}_2(K)$ be a matrix whose 
    eigenvalues $a,b$ lie in a quadratic extension $L$ of $K$.
    Then there is a maximal order that is invariant under 
    conjugation by $\mathtt{r}$ if and only 
    if the fractional ideal $a\mathcal{O}_L$ has the form 
    $J\mathcal{O}_L$ for some fractional ideal $J$ in $K$.
\end{Cor}

\begin{proof}
    Since $\{a,b\}$ is a Galois orbit, 
    the condition that $\frac ab$ is a unit,
    or equivalently that $a\mathcal{O}_L=b\mathcal{O}_L$, 
    is equivalent to the
    Galois invariance of the ideal $a\mathcal{O}_L$. 
    For a Galois invariant
    ideal to be defined over $K$ it suffices
    to check that it have even valuation
    at the ramified places. The result follows.  
\end{proof}

\section{Indecomposable abelian representations}\label{S9B}

In this section we consider representations 
$\varphi:G\rightarrow
\mathrm{GL}_2(K)$, that are not decomposable over the 
number field
$K$, but they  are decomposable over an extension $L/K$. 
In particular $G$ is abelian. Over $L$, 
the representation $\varphi$ must be isomorphic 
to the direct sum of 
two one-dimensional representations, 
$\chi_1$ and $\chi_2$, since
$L$ is a field of characteristic $0$, and therefore the group
algebra $L[G]$ is semisimple. For every element $g\in G$, the
sum $\chi_1(g)+\chi_2(g)$ is the trace of 
$\varphi(g)$, and therefore
it belongs to $K$. For the same reason, the determinant
$\chi_1(g)\chi_2(g)$ is in $K$, whence 
$\chi_1(g)$ and $\chi_2(g)$
belong to an extension $L_g/K$ of degree $1$ or $2$. If $L_g=K$
for all $g$, then $\chi_1$ is defined over $K$ contrary to
the assumption. We conclude that $L_g\neq K$ for at 
least one $g\in G$, 
so that the eigenvalues
$\chi_1(g)$ and $\chi_2(g)$ of the 
matrix $\mathtt{r}=\varphi(g)$ 
are different, and satisfy the same 
quadratic equation $f(X)=0$ with 
coefficients in $K$. In particular,
$f\big(\phi(g)\big)=0$.
We conclude that $K[\mathtt{r}]\cong K[X]/(f)$
is both, a field, and a maximal abelian subalgebra of
$\mathbb{M}_2(K)$. This implies that $\varphi(G)$ 
is contained in a field,
and therefore it is a cyclic group. 
We can assume therefore that
$G=\langle g_0\rangle$ is cyclic, which we do in the sequel.
Furthermore, the matrix $\mathtt{r}_0=\varphi(g_0)$ 
has no eigenvalues
in the field $K$, but they are
contained in the quadratic extension $L=L_\varphi$ of $K$.
Since $\mathtt{r}_0^{|G|}$ is the identity matrix,
the eigenvalues 
of $\mathtt{r}_0$ are roots of unity, and each generates
the extension $L/K$. Moreover, if 
$\mathrm{Gal}(L/K)=\langle\sigma\rangle$,
then $\sigma(\chi_1)=\chi_2$. 

Since $\chi_1(g_0)$ and 
$\chi_2(g_0)$ are roots of unity, the trace
$\tau=\chi_1(g_0)+\chi_2(g_0)$ and the determinant
$\delta=\chi_1(g_0)\chi_2(g_0)$ belong to the ring of integers
$\mathcal{O}_K$. It does not hurt, therefore, to assume
$\mathtt{r}_0=\sbmattrix 0\delta{-1}\tau$.
The order $\mathcal{O}_K[\mathtt{r}_0]\subseteq
\mathbb{M}_2(K)$ is contained in the maximal order
$\mathfrak{H}_\varphi$ of the algebra 
$K[\mathtt{r}_0]$. By the general 
theory, the localization 
$\mathfrak{H}_{\varphi,\nu}$ at every place
$\nu$ is contained in the maximal order corresponding to
every vertex in a path $\mathfrak{p}(\nu)$.
The path $\mathfrak{p}(\nu)$ consists of a unique vertex, 
at places $\nu$ that are inert for $L/K$, 
an edge, at ramified places $\nu$,
or a maximal path, at split places $\nu$. 
These are called the stem
orders of the representation $\varphi$. 
As in \S\ref{S6}, for every maximal 
order $\mathfrak{D}$ containing $\mathcal{O}_K[\mathtt{r}_0]$,
its anchor order is defined as 
the unique stem order at minimal local distance
at all places. Also as before, the set of 
maximal orders containing
$\mathcal{O}_K[\mathtt{r}_0]$ with the same 
anchor order is a product
of side branches. The centralizer of the 
representation $\varphi$,
in this case, is the algebra $K[\mathtt{r}_0]$. The group
$T=K[\mathtt{r}_0]^*$ acts by conjugation on 
the set of stem orders.
If we choose a vector $\mathbf{v}\in K^2$, then
$\mathtt{h}\mapsto \mathtt{h}\left(\mathbf{v}\right)$ 
defines a linear 
isomorphism between $K[\mathtt{r}_0]$ and $K^2$ 
which we use to identify
both spaces in the remainder of this section. Then stem orders
have the form $\mathfrak{D}_\Lambda$ for a lattice $\Lambda$
satisfying $\mathfrak{H}_\varphi\Lambda\subseteq\Lambda$. 
In other
words, $\Lambda$ must correspond to a fractional ideal of the
algebra $K[\varphi]\cong L$. Recall that $\mathfrak{D}_\Lambda=
\mathfrak{D}_{\Lambda'}$ precisely when $\Lambda'=J\Lambda$
for some fractional ideal $J\subseteq K$. We conclude that 
the set of stem orders is naturally in correspondence with 
the group
$\mathcal{I}_L/\mathcal{I}_K$. The $T$-action by conjugation
on orders correspond to the natural action of $L^*$
on $\mathcal{I}_L/\mathcal{I}_K$ by multiplication. 
Hence, the orbit
set corresponds to the quotient 
$\mathcal{I}_L/\mathcal{P}_L\mathcal{I}_K$.
In particular the group of units 
$\mathcal{O}_L^*$ acts trivially
on the set of stem orders, but it might 
act non-trivially on side 
branches.

However, this case does present some additional difficulties to
a straightforward generalization of our results for
the split case. Firstly, it is not guaranteed 
that the set of stem 
orders contain orders in all conjugacy classes. In other words,
the maximal order $\mathfrak{H}_\varphi$ might turn up to be a 
selective order. This has been studied in the works of
C. Chevalley, see \cite{che36}, or  T. Chindburg and 
E. Friedman, see \cite{chi99}. For notational convenience, we 
use Lemma \ref{lsel} below. It is a particular case
of \cite[Th. 1]{A03}. Note that the spinor class 
field $\Sigma$,
for the particular case of matrix algebras,
is described in \textit{Ex. 1 Continued}, at the end of
\cite[\S2]{A03}.

\begin{Lem}\label{lsel}
\textnormal{\cite[Th. 1]{A03}.}
Let $\Sigma$ be the largest sub-extension of 
exponent $2$ of the
wide Hilbert class field of $K$. Then the ring of integers 
$\mathfrak{H}$ of a quadratic extension $L$ of $K$ embeds 
into every maximal order in the matrix algebra 
$\mathbb{M}_2(K)$, except when
$L\subseteq\Sigma$. In the later case 
$\mathfrak{H}$ is contained in the maximal orders
belonging to precisely one half of
all the conjugacy classes.\qed
\end{Lem}

Nevertheless, we note that
the order $\mathcal{O}_K[\mathtt{r}_0]$ does have an integral 
representation, namely $\varphi$. If we assume that
$\mathtt{r}_0=\sbmattrix 0\delta{-1}\tau$ as before, 
then the set of
local maximal orders containing 
$\mathtt{r}_0$ is in correspondence
with the invariant vertices of the Moebius transformation
$\mu(z)=\frac{\delta}{\tau-z}$, whose fixed points in
$\mathbb{P}^1$ are $\xi_1=\chi_1(g_0)$ and 
$\xi_2=\chi_2(g_0)$.
At places $\nu'$ splitting the extension $L/K$, the local
maximal orders containing $\mathtt{r}_0$ are those in a tubular
neighborhood of the maximal path with visual limits $\xi_1$
and $\xi_2$. The width of the neighborhood is the valuation
$\nu'(\xi_1-\xi_2)$, since $\xi_1$ and $\xi_2$ are 
also the eigenvalues of $\mathtt{r}_0$. We conclude that the
vertex $v_0^{[0]}$ corresponding to the ball $B_0^{[0]}$
is a leaf of the branch $\mathfrak{s}(\mathtt{r}_0)$ whenever
$\nu'(\xi_1-\xi_2)>0$. If 
$\nu'(\xi_1-\xi_2)=0$, then $\mathfrak{s}(\mathtt{r}_0)$
is a maximal path containing $v_0^{[0]}$. A similar
result follows at non-split places by passing to a quadratic 
extension.

Now assume that $K=K_0:=\mathbb{Q}[\tau,\delta]$ is the field
of definition of the representation $\varphi$. Equivalently,
$K$ is the invariant subfield, in
$\mathbb{Q}(\xi_1)=\mathbb{Q}(\xi_2)$, of the automorphism
$\sigma$ taking $\xi_1$ to $\xi_2$. Note that $\sigma^2$
must be the identity.
The ring of 
integers of $L=\mathbb{Q}[\xi_1]$ is $\mathbb{Z}[\xi_1]$
by Lemma \ref{lc1}.
Next result follows from the isomorphism 
$\mathcal{O}_K[\mathtt{r}_0]\cong\mathcal{O}_K[\xi_1]$:
\begin{Prop}
    In the preceding notations, if 
    $K=K_0$ is the field of definition of the
    representation $\varphi$, then 
    $\mathcal{O}_K[\mathtt{r}_0]=
    \mathfrak{H}_\varphi$, and therefore the branch 
    $\mathfrak{s}_{\nu'}(\varphi)$ is a path at every 
    local place
    $\nu'$. \qed
\end{Prop}

\begin{ex}
The previous result fails to extend to a general extension $K$ 
of $K_0$. In fact, Example \ref{ec1} shows that the branch 
$\mathfrak{s}_{\nu'}(\varphi)$ is a non-trivial tubular 
neighborhood of a path when $K=\mathbb{Q}[\sqrt{-5}]$,
$\nu=\nu_{\mathbf{2}_1}$ is the dyadic place,
and $\varphi$ is the unique two-dimensional faithful 
representation of $C_4$ defined over $K$.

\end{ex}

\begin{Prop}\label{p42}
In the preceding notations, if 
    $K=K_0$ is the field of definition of the
    representation $\varphi$, then
    there are $h_{L/K}$ conjugacy classes of
    $\mathcal{O}_K$-representations that are 
    $K$-conjugate to $\varphi$.  
\end{Prop}

\begin{proof}
    Recall that $\varphi:G\rightarrow\mathbb{M}_2(K)^*$ 
    turns $K^2$ into
    a $K[\mathtt{r}_0]$-module, that must be isomorphic to 
    $L$ as an $L$-module, when we identify the subalgebra
    $K[\mathtt{r}_0]$ with $L$.
    In particular, the $\mathcal{O}_K$-lattices that are
    $\mathcal{O}_K[\mathtt{r}_0]$-invariant correspond 
    precisely with
    the fractional $\mathcal{O}_L$-ideals $I\subseteq L$. 
    The latter are
    free as $\mathcal{O}_K$ modules precisely when the ideal
    $N_{L/K}(I)J$, as defined in Lemma \ref{lp1}, is principal.
    It follows that the $T$-orbits of these modules 
    correspond to
    the elements in a coset of $\mathcal{G}_{L/K}$ 
    in $\mathcal{G}_L$,
    provided that $J=J(L/K)$ is in the same class 
    as the norm in
    $\mathcal{I}_K$ of an ideal in $\mathcal{I}_L$. 
    Otherwise, none of
    these modules would be free, but we know 
    this is not the case
    since $G$ does have an integral representation.
\end{proof}

\begin{Cor}\label{cor831}
   In the preceding notations,
   the number of $T$-conjugacy classes of maximal orders 
    that are global conjugates of $\mathbb{M}_2(\mathcal{O}_K)$
    and contain $\varphi(G)$ is $\frac{h_{L/K}}{t}$,
    where $t$ is the order of the image of 
    $\mathcal{G}_K(2)$ in 
    $\mathcal{G}_{L/K}\subseteq\mathcal{G}_L$,
    under the map $[J]\mapsto[J\mathcal{O}_L]$.
\end{Cor}

\begin{proof}
    Identify $K^2$ with $L$ as before. 
    Denote by $\mathcal{I}_K(2)$
    the pre-image in $\mathcal{I}_K$ of $\mathcal{G}_K(2)$.
      An element $\lambda$ of 
      $L^*\cong K[\mathtt{r}_0]^*=T$ fixes one (and 
      hence every) order in the branch of 
      $\mathfrak{H}_\varphi$, locally
      everywhere, if and only if it generates a 
      fractional ideal of the  form 
      $J\mathcal{O}_L$, for some fractional ideal 
      $J\subseteq K^*$, by Cor.
      \ref{c741}. 
      Note that $J^2$ is generated by 
      $N_{L/K}(\lambda)$, so that
      $J\in\mathcal{G}_K(2)$. 
      This means that for every order $\mathfrak{D}$ containing
      $\mathtt{r}_0$ we have 
      $\Theta_{\mathfrak{D}}=
      \mathcal{G}_K(2)\cap\mathcal{P}_L$, when we 
      identify $\mathcal{I}_K$ with its natural 
      image in $\mathcal{I}_L$ as
      usual. This implies that $|\Theta_{\mathfrak{D}}|=
      \frac{h_K(2)}t$, and therefore each maximal 
      order corresponds
      to  $\frac{h_K(2)}{h_K(2)/t}=t$ representations by Lemma
      \ref{Lemcc}. The result follows.
\end{proof}

\begin{ex} 
    For the group considered in Ex. \ref{e25},
    the class group $\mathcal{G}_K$ is trivial, 
    and therefore $t=1$.
\end{ex}

\begin{rem}
    We note that the proof of Prop. \ref{p42} 
    extends almost verbatim
    to $n$-dimensional irreducible representations.
\end{rem}

For a general extension, however, the 
order $\mathcal{O}_K[\mathtt{r}_0]$
does not need to equal the maximal order $\mathfrak{H}_\varphi$
of its spanned algebra.  This means that there is a 
finite set $\Pi$
of places $\nu$ where $\mathcal{O}_\nu[\mathtt{r}_0]\neq
\mathfrak{H}_{\varphi,\nu}$, according to \ref{it2}. 
At those places, the local branch
$\mathfrak{s}_{\nu}(\mathtt{r}_0)$ has foliage vertices. 
 Contrary to the split case, 
 at non-split places there 
are only one or two local stem orders. However, 
side branches still look as those depicted
in Figure \ref{f2p}.

\begin{Lem}
    Side branches corresponding to different anchor 
    orders can be 
    identified, non canonically, in a way that 
    commutes with the
    action of $\mathcal{O}_K^*$.
\end{Lem}

\begin{proof}
This is proved locally, as for Lemma \ref{lsd1}. 
At split places
we can apply a conjugation and assume 
that the stem is indeed the
standard apartment $\mathfrak{a}_s$. Then the centralizer 
is the group $T\subseteq \mathrm{GL}_2(K)$ of diagonal 
matrices and the proof goes precisely as for Lemma \ref{lsd1}. 
At inert places the stem order is unique, 
so there is nothing to 
prove. At ramified places the uniformizer of the local algebra
$\mathfrak{L}_\nu=\mathcal{O}_\nu[\mathtt{r}_0]\cong L_\omega$
transpose, by conjugation, the two maximal orders containing 
the maximal order $\mathfrak{H}_{\varphi,\nu}$, 
while it centralizes
the representation, so it commutes with the action of the group
of units $\mathcal{O}_L^*$, as they act via the identification
$\mathcal{O}_L^*=\mathfrak{H}_\varphi^*\subseteq
\mathfrak{L}_\nu^*$. 
Now the result follows as before.
\end{proof}

\begin{proof}[Proof of Theorem \ref{t6}]
It is immediate that $I'\mapsto U_{I'}$
preserves inclusions, and the condition
defining the group $U$ guarantees
that its leaves invariant the stem orders,
as we saw in the proof of 
Cor. \ref{cor831} above.
This proves that $U$ acts on the product
of side branches.
From Lemma \ref{l74}, we conclude both,
$\frac{\sigma(\lambda)}\lambda\in\mathcal{O}_K^*$ and
$\frac{\sigma(\lambda)}\lambda\equiv1\mod I[L/K]$, 
so that $U_{(1)}=U$.

Locally we write 
$\mathcal{O}_\nu[\mathtt{r}_0]=
\mathfrak{H}_{\varphi,\nu}^{[m_\nu]}=
\mathcal{O}_K\mathtt{1}+\pi_\nu^{m_\nu}
\mathfrak{H}_{\varphi,\nu}$,
where $m_\nu\geq 0$ equals the largest distance
a vertex in the side branch can be from its anchor vertex. 
Globally, we have $\mathcal{O}_K[\mathtt{r}_0]=
\mathcal{O}_K\mathtt{1}+I\mathfrak{H}_\varphi$, 
where $I_\varphi=
\prod_{\nu\in\Pi}P[\nu]^{m_\nu}$
is the smallest ideal valued distance from 
any order corresponding to a vertex $w$ in 
$\mathbb{B}$ to the anchor order 
$\mathfrak{D}_v$. In other words,
all ideal valued distances $I'=I_{(v,w)}$ divide $I_\varphi$. 
The remainder of the proof is very similar to
that of Theorem \ref{t5} with one mayor difference. 
The contribution of all vertices
is no longer the same. The computation of
the group $\Theta_{\mathfrak{D}}$, in this case,
does depend on the distance ideal $I'$
from an order $\mathfrak{D}$ to the 
corresponding anchor order. 

By Cor. \ref{c741}, an element $\lambda$ of 
$L^*\cong K[\mathtt{r}_0]^*=T$ fixes one (and 
hence every) anchor order if it generates a 
fractional ideal of the 
form $J\mathcal{O}_L$ 
for some fractional ideal $J\subseteq K^*$. 
This is the reason why the contribution of
a stem order is $t=\frac{h_K(2)}{\mu_{(1)}}$.
It suffices, therefore to prove that the contribution
of a foliage order at distance $I'$ from its stem order
is $t_{I'}=\frac{h_K(2)}{\mu_{I'}}$. For this, it
suffices to prove that an element $\lambda\in U$
fixes every vertex at distance $I'$ from the anchor vertex
precisely when $\lambda\in U_{I'}$. This is done locally.
As usual we fix an identification between side branches
of different stem orders.

At places $\nu$ splitting the extension $L/K$, this is a direct 
application of Lemma \ref{l61}, since the image of 
$\lambda$ in the stabilizer $\mathfrak{L}^*$
is a matrix with eigenvalues $\lambda$ and $\sigma(\lambda)$. 
 When $\nu$ is inert for the 
extension $L/K$, the result follows by passing to a suitable quadratic 
extension and observing that the branch of $\mathfrak{H}_{\varphi,\nu}$
in the extension is a maximal path containing the unique vertex
in the branch of $\mathfrak{H}_{\varphi,\nu}$ in $K$, so for
vertices in the Bruhat-Tits tree at $K$ their distance to either
branch (over $K$ or $L$) is the same. In all that follows we assume 
that $L/K$ ramifies at $\nu$. 

Let $\omega$ the only place of $L$ lying above $\nu$. Then
$\mathcal{O}_{L,\omega}=\mathcal{O}_{K,\nu}[\pi_\omega]$.
Recall that the distance from either vertex, $w_1$ 
or $w_2$, in the
branch of $\mathfrak{H}_{\varphi,\nu}$, to the path
$\mathfrak{p}\subseteq\mathfrak{t}$, as in Lemma
\ref{l61d}, is $\kappa+1=\omega\left(\sigma(\pi_\omega)-\pi_\omega\right)$.
A vertex $w\in\mathfrak{t}$,
at distance $m$ from the closest vertex in
$\{w_1,w_2\}$, is at a distance of
$2m+\kappa+1$
from the stem $\mathfrak{p}$. The $2$ appears because the 
extension $L_\omega/K_\nu$ is ramified.
Since $I[L/K]_\omega=
\left(\sigma(\pi_\omega)-\pi_\omega\right)=
P[\omega]^{\kappa+1}$,
by definition, and $P[\omega]^2=P[\nu]\mathcal{O}_L$
since the extension is ramified, we conclude
that the valuation of $\lambda-\sigma(\lambda)$
is large enough for the corresponding matrix
to fix the vertices at distance $m$ from
the closest vertex in $\{w_1,w_2\}$ precisely when
$\frac{\sigma(\lambda)}{\lambda}\equiv1\mod 
P[\nu]_\nu^mI[L/K]_\omega$.
The result follows.
\end{proof}

\begin{ex}\label{e89}
    Let $G=C_4=\langle g\rangle$, and let $
    K=\mathbb{Q}[\sqrt{-5}]$. Consider
    the representation defined by 
    $\varphi(g)=\sbmattrix 01{-1}0$. Over the field 
    $K_0=\mathbb{Q}$, the maximal orders containing this
    matrix has only two possible completions at the place $2$.
    They correspond to the vertices $u=v_0^{[0]}=v_1^{[0]}$ 
    and $w=v_1^{[1]}$.
    Since $K/K_0$ is ramified at $2$, there exists an intermediate
    vertex $v=v_1^{[1/2]}$ at the dyadic place $\mathbf{2}_1$.
    Note that the ball names correspond to a valuation
    $\nu=\nu_{\mathbf{2}_1}$ satisfying $\nu(2)=1$,
    and therefore the valuation of a uniformizer is $\frac12$.
    We use this normalized valuation in this example
    and next for simplicity.
    Furthermore, Since $L=K[\sqrt{-1}]=
    \mathbb{Q}[\sqrt{-1},\sqrt5]$ is an unramified quadratic
    extension of $K$, i.e., it is contained in the Hilbert class
    field, it is selective by Lemma \ref{lsel}. 
    Since the class number
    of $K$ is $2$, the order $\mathfrak{H}_\varphi$ embeds into
    the maximal orders in precisely one conjugacy class.
    Given that:
    \begin{itemize}
        \item the invariant points in $\mathbb{P}^1$ of
    the Moebius transformation $\iota(z)=\frac{-1}z$ are
    the roots $\pm\sqrt{-1}$, and the smallest ball containing
    both is $B_{\sqrt{-1}}^{[1]}$,
        \item the ball $B_{\sqrt{-1}}^{[1]}$ corresponds to
    the only vertex of the path joining the roots $\pm\sqrt{-1}$
    that is defined over the field $K$,
        \item $\nu(-\sqrt{-1}-\sqrt{-1})=\nu(2)=1$, and 
    $\nu(\sqrt{-5}-\sqrt{-1})=1$,
    \end{itemize}
    we conclude that the 
    branch $\mathfrak{s}_K(\varphi)$ is a ball centered at
    the vertex $v_c=v_{\sqrt{-1}}^{[1]}$, as depicted in Figure \ref{twoballs}.A.
\begin{figure}
\unitlength 1mm 
\linethickness{0.4pt}
\ifx\plotpoint\undefined\newsavebox{\plotpoint}\fi 
\[
\begin{picture}(38,20)(-6,32)
\put(0,36){A}
\put(15.2,43.2){$\bullet$}\put(15.2,51.8){$\circ$}
\put(16,44){\line(0,1){8}}
\multiput(16,44)(0.086,-.05){80}{\line(1,0){.3}}
\multiput(16,44)(-0.086,-.05){80}{\line(-1,0){.3}}
\put(7.7,38.7){$\circ$}\put(22.3,38.7){$\circ$}
\multiput(15.4,52.6)(-0.086,.05){40}{\line(-1,0){.3}}
\multiput(16.6,52.6)(0.086,.05){40}{\line(1,0){.3}}
\put(10.8,53.8){$\bullet$}\put(19.6,53.8){$\bullet$}
\put(8.6,38.8){\line(0,-1){4}}\put(23.2,38.8){\line(0,-1){4}}
\multiput(8.1,39.7)(-0.086,.05){40}{\line(-1,0){.3}}
\multiput(23.7,39.7)(0.086,.05){40}{\line(1,0){.3}}
\put(3.8,40.8){$\bullet$}\put(26.2,40.8){$\bullet$}
\put(7.7,34){$\bullet$}\put(22.3,34){$\bullet$}
\multiput(10.6,54.6)(-0.9,0){12}{\line(1,0){.3}}
\multiput(20.6,54.6)(0.9,0){12}{\line(1,0){.3}}
\put(-3.4,54){${}^\infty\star$}\put(31,54){$\star^1$}
\multiput(15.6,44.6)(-0.8,0.1){12}{\line(1,0){.3}}
\multiput(15.6,44.6)(0.8,0.1){12}{\line(1,0){.3}}
\put(-4.5,45.5){${}^{-\sqrt{-1}}\star$}
\put(25,45.5){$\star^{\sqrt{-1}}$}
\put(15.2,53.8){$v$}\put(10.8,55.8){$v_0$}
\put(19.6,55.8){$v_1$}\put(15.2,40.8){$v_c$}
\end{picture}
\qquad\qquad\qquad\qquad
\begin{picture}(32,20)(0,32)
\put(0,36){B}
\put(20,44){$\circ$}
\put(21,45.4){\line(0,1){8}}\put(21.6,45){\line(1,0){8}}
\put(20.4,45){\line(-1,0){8}}\put(21,44.4){\line(0,-1){8}}
\put(12,44){$\bullet$}\put(28,44){$\bullet$}
\put(20,52){$\bullet$}\put(20,36){$\bullet$}
\put(18.5,46){$v$}\put(12,47){$v_0$}
\put(28,47){$v_1$}
\multiput(12,45)(-0.9,0){12}{\line(1,0){.3}}
\multiput(28,45)(0.9,0){12}{\line(1,0){.3}}
\put(-2.4,44){${}^\infty\star$}\put(38,44){$\star^1$}
\multiput(21,45.8)(0.45,0.45){12}{\line(1,0){.3}}
\multiput(21,44)(0.45,-0.45){12}{\line(1,0){.2}}
\put(26.5,51){$\star^{-\omega}$}
\put(26.5,37){$\star_{-\omega^2}$}
\end{picture}
\]
\caption{Two balls in the Bruhat-Tits tree. }\label{twoballs}
\end{figure}
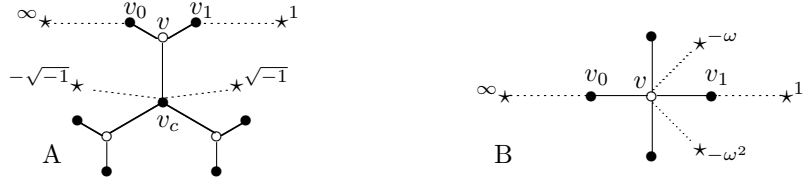
    By parity consideration, the anchor order $\mathfrak{D}_c$
    of $\mathfrak{D}_0$, which corresponds to the vertex $v_c$,
    is isomorphic to $\mathbb{M}_2(\mathcal{O}_K)$.
    We conclude that all stem orders for this representation
    are isomorphic to $\mathbb{M}_2(\mathcal{O}_K)$, so there
    are no representation associated to the 
    white circles in the picture.

    As for the number of representations corresponding to
    each vertex, we note that
    $$U=\left\langle1+\sqrt{-1}\right\rangle 
    K^*\mathcal{O}^*_L,$$
    since $u=1+\sqrt{-1}$ is a generator 
    in $L$ of the extension
    of a non-principal fractional ideal in $K$. It follows that
    $\mu_{(1)}=2$, and therefore $t_{(1)}=1$. 
    On the other hand,
    we have $\frac u{\sigma(u)}=\sqrt{-1}=1+\epsilon$, where
    $\epsilon$ is a uniformizer. This is an element of order
    $2$ in $(\mathcal{O}_L/2\mathcal{O}_L)^*$. This proves that
    $u\notin U_{(2)}$. We need to check, however, that no 
    element of the form $uu'$, with $u'\in\mathcal{O}_L^*$, 
    is in $U_{(2)}$. Note that $\mathcal{O}_L^*
    =\langle\sqrt{-1},\eta_0\rangle$,
    where $\eta_0=\frac{1+\sqrt{5}}2$ is a fundamental unit 
    of $L$. In fact, $\eta_0$ has order $3$ modulo $2$, 
    as it is contained
    in a quadratic field where $2$ is inert, while $2$ 
    divides neither $\eta_0$ nor $\eta_0-1$. We note that 
    both $\frac {\sqrt{-1}}{\sigma(\sqrt{-1})}=-1$ and
    $\frac {\eta_0}{\sigma(\eta_0)}=\eta_0^2$ are
    contained in a subgroup of order $3$ in 
    $(\mathcal{O}_L/2\mathcal{O}_L)^*$. We conclude
    that $U_{(2)}\subseteq K^*\mathcal{O}_L^*$. In particular, 
    $\mu_{(2)}=1$, so  that the black leaves in Fig. 
    \ref{twoballs}.A contribute
    with two representations each. We conclude that the number
    of integral representations is $13h_{L/K}$.
\end{ex}

\begin{ex}\label{e89b}
    Let $G=C_6=\langle g\rangle$, and let 
    $K=\mathbb{Q}[\sqrt{-15}]$. Consider
    the representation defined by 
    $\varphi(g)=\sbmattrix 01{-1}1$. Over the field 
    $K_0=\mathbb{Q}$, the maximal orders containing this
    matrix have only two possible completions at the place $3$.
    They correspond to the vertices $u=v_1^{[0]}$ 
    and $w=v_1^{[1]}$.
    Since $K/K_0$ is ramified at $3$, there 
    exists an intermediate
    vertex $v_1^{[1/2]}$ at the corresponding
    place $\mathbf{3}_1$. We normalize the valuation
    $\nu=\nu_{\mathbf{3}_1}$ in a way that $\nu(3)=1$.
    Furthermore, since $L=K[\sqrt{-3}]=
    \mathbb{Q}[\sqrt{-3},\sqrt5]$ is an unramified quadratic
    extension of $K$, it is selective by Lemma \ref{lsel}. 
    Since the class number
    of $K$ is $2$, the order $\mathfrak{H}_\varphi$ embeds into
    the maximal orders in precisely one conjugacy class.
    The invariant points in $\mathbb{P}^1$, 
    of the corresponding
    Moebius transformation $\iota_0(z)=\frac1{1-z}$ are
    $-\omega$ and $-\omega^2$, where 
    $\omega=\frac{-1+\sqrt{-3}}2$ is a primitive cubic root of
    $1$. These fixed points span the ball
    $B_{-\omega}^{[1/2]}$, since 
    $\nu(-\omega+\omega^2)=
    \nu(1-\omega)=\frac12\nu(3)=\frac12$.
    Again using $\nu(1-\omega)=\frac12$, we conclude that
    $B_{-\omega}^{[1/2]}=B_1^{[1/2]}$. it follows that 
    $v=v_1^{[1/2]}$ is actually the center of the ball
    $\mathfrak{s}_K(\varphi)$. This proves that, in this case,
    no stem order is isomorphic to 
    $\mathbb{M}_2(\mathcal{O}_K)$.
    In other words, the order 
    $\mathfrak{H}_\varphi$ has no integral representation.
\end{ex}

\begin{ex}\label{e911}
    Let $G=C_n$, where $n=p^t$ is a prime power. 
    Let $\nu$ be the only
    place of $K_0$ lying over $p$.
    Assume $K/K_0$  has even ramification degree at some, 
    but not all, the places over $\nu$. For instance, 
    this is the case
    for a degree-$4$ extension,
    if there exist an intermediate field 
    $K_0\subseteq F\subseteq K$
    such that $F/K_0$ splits at $\nu$, and then $K/F$ ramifies
    at one of the places over $\nu$, say $\nu_1$, 
    but not the other,
    say $\nu_2$. Then, reasoning as in the preceding 
    examples, it can be shown that the $n$-th root of 
    unity do not generate the maximal
    order at $\nu_1$, but still generates a ramified extension
    over $\nu_2$. This proves that Theorem \ref{t6} applies,
    since $\mathfrak{H}_\varphi$ cannot be selective, 
    but non-trivial side branches do appear at $\nu_1$.
\end{ex}

\begin{ex}\label{e912}
    Let $G=C_n$, where $n=p^t$ is an odd prime power, and
    let $\nu$ be as above. Then, if
    $K/K_0$ has even ramification degree at every place 
    lying over $\nu$, we can show that $L/K$ is 
    unramified. This shows
    that $\mathfrak{H}_\varphi$ is selective,
    so Theorem \ref{t6} does not apply. 
    However, this fails for $p=2$,
    since $\sqrt{-1}$ generates a ramified extension of
    $\mathbb{Q}_2(\sqrt2)$.
\end{ex}

\section{Absolutely irreducible representations}\label{S9}

For an absolutely irreducible representation 
$\psi:G\rightarrow\mathrm{GL}_2(K)$, the order
$\mathfrak{H}$ spanned by $\psi(G)$ is an order
of maximal rank, and therefore it coincides with
the maximal order $\mathbb{M}_2(\mathcal{O}_K)$
outside a finite set $\Pi$ of exceptional places.
The local-global principle for lattices (\textbf{LGPL}),
as described in \S2, shows that the maximal orders 
containing $\psi(G)$ are in correspondence with the
vertices $v$ in the product 
$\mathbb{S}(\psi)=\prod_{\nu\in\Pi}\mathfrak{s}_\nu(\psi(G))$, where
$\mathfrak{s}_\nu(\psi(G))$ denotes the local branch at $\nu$.
This is contained in the product 
$\prod_{\nu\in\Pi}\mathfrak{t}_\nu$ of Bruhat-Tits trees.
Every tree $\mathfrak{t}_\nu$ has a distinguished
vertex $v_{0,\nu}$ corresponding to the ball
$B_0^{[0]}$, or equivalently, to the maximal order
$\mathbb{M}_2(\mathcal{O}_\nu)$. 
We can define the restricted product
as the subset $\prod_{\nu\in\Pi}'\mathfrak{t}_\nu$, where
every coordinate outside a finite set of places coincide
with the distinguished vertex $v_{0,\nu}$. 
This restricted product is
a cell complex having a distinguished vertex $v_0$, 
for which every coordinate is $v_{0,\nu}$.
This vertex belong to $\mathbb{S}(\psi)$ precisely when
$\psi$ is has integral coefficients.

\begin{proof}[Proof of Thm. \ref{t7}]
    It suffices to note that for an absolutely irreducible
    representation $\psi_0$ we have $\mathfrak{L}_{\psi_0}=K$,
    and the determinant of every element $\varepsilon\in K^*$
    is $\varepsilon^2$, so the corresponding
    class in $\mathcal{G}_K(2)$ is trivial.
    This shows that $\Theta_{\mathfrak{D}}$ is trivial
    for every maximal order $\mathfrak{D}$ corresponding to a vertex in
    the branch.
    Now the result is a straightforward application of
    Prop. \ref{percc} and Lemma \ref{Lemcc}.
\end{proof}

\begin{Corth}
An absolutely irreducible representation 
$\psi_0:G\rightarrow\mathrm{GL}_2(K)$ is 
$\mathrm{GL}_2(K)$-conjugate to an integral 
representation precisely when the product 
$\mathbb{S}(\psi)$
has a vertex whose Artin distance to $v_0$ is trivial.\qed
\end{Corth}

\begin{Corth}
If an absolutely irreducible representation 
$\psi:G\rightarrow\mathrm{GL}_2(K)$ is 
$\mathrm{GL}_2(K)$-conjugate to two or more integral 
representations that are not 
$\mathrm{GL}_2(\mathcal{O}_K)$-conjugates, then
for any positive constant $M$ there exists a field extension
$E/K$ such that $\psi$ is conjugate over $E$
to more that $M$ representations with values in 
$\mathrm{GL}_2(\mathcal{O}_E)$
that are pairwise non-conjugates in this group.
\end{Corth}

\begin{proof}
    Let $\nu$ be a place of $K$ such that 
    $\mathfrak{s}_\nu(\psi)$ has two 
    different vertices $v$ and $v'$. Let $\mathfrak{D}$ and 
    $\mathfrak{D}'$ be the corresponding local orders. 
    Write $P[\nu]\mathcal{O}_E=
    P_1^{e_1}\cdots P_n^{e_n}$ where each $P_i$ 
    is a prime ideal
    of $E$. Write $v_i$ for the vertex in the 
    Bruhat-Tits tree for
    $E_{P_i}$ corresponding to $\mathcal{O}_{P_i}\mathfrak{D}$,
    and define $v_i'$ analogously.
    Then the branch of $\psi$ at $\nu_{P_i}$, 
    as a place of $E$,
    contain every vertex in the path from $v_i$ to $v'_i$,
    and their number is $e_i+1$. The result follows by
    choosing an extension where 
    $\prod_{i=1}^n(e_i+1)$ is large enough. 
\end{proof}

Next, we present two examples that illustrate how branches,
and their corresponding representations, are computed.

\begin{ex}\label{e92}
Consider the representation 
$\psi:D_4\rightarrow\mathrm{GL}_2(\mathbb{Q}
\big(\sqrt{-5})\big)$, where
$D_4$ is the dihedral group of order $8$ with the presentation
$\langle a,b| a^4=b^2=e, bab=a^{-1}\rangle$, where
$$\psi(a)=\bmattrix 01{-1}0,\qquad\psi(b)=\bmattrix 0110.$$
Note that the maximal orders containing $\psi(a)$ appear 
in Example \ref{e89} above. 
For the sake of simplicity, we keep the
normalization in that example.
Since $\psi(b)$ has eigenvalues $1$ and $-1$, it is
contained in the maximal orders corresponding 
to the vertices at distance
$\nu_{\mathbf{2}_1}\big(1-(-1)\big)=1$ 
of the path connecting the
fixed point $1$ and $-1$ of the Moebius 
transformation $\mu(z)=\frac1z$.
Note, however, that this is twice
the valuation of a uniformizer.
\begin{figure}
\unitlength 1mm 
\linethickness{0.4pt}
\ifx\plotpoint\undefined\newsavebox{\plotpoint}\fi 
\[
\begin{picture}(38,12)(-6,40)
\put(15.2,43.2){$\bullet$}\put(15.2,51.8){$\circ$}
\put(16,44){\line(0,1){8}}
\multiput(16,44)(0.688,-.4){10}{\line(1,0){.3}}
\multiput(16,44)(-0.688,-.4){10}{\line(-1,0){.3}}
\put(7.7,38.7){$\circ$}\put(22.3,38.7){$\circ$}
\multiput(15.4,52.6)(-0.086,.05){40}{\line(-1,0){.3}}
\multiput(16.6,52.6)(0.086,.05){40}{\line(1,0){.3}}
\put(10.8,53.8){$\bullet$}\put(19.6,53.8){$\bullet$}
\multiput(20.6,54.6)(0.9,0.35){12}{\line(1,0){.2}}
\multiput(10.6,54.6)(-0.9,0){12}{\line(1,0){.3}}
\multiput(20.6,54.6)(0.9,0){12}{\line(1,0){.3}}
\put(-3.4,54){${}^\infty\star$}\put(31,54){$\star^1$}
\put(31,58){$\star^{-1}$}
\multiput(15.6,44.6)(-0.8,0.1){12}{\line(1,0){.3}}
\multiput(15.6,44.6)(0.8,0.1){12}{\line(1,0){.3}}
\put(-4.5,45.5){${}^{-\sqrt{-5}}\star$}
\put(25,45.5){$\star^{\sqrt{-5}}$}
\put(15.2,53.8){$v$}\put(10.8,55.8){$v_0$}
\put(19.6,55.8){$v_1$}\put(15.2,40.8){$v_c$}
\end{picture}
\]
\caption{The vertices defining representations in Example
\ref{e92}. }\label{oneballs}
\end{figure}
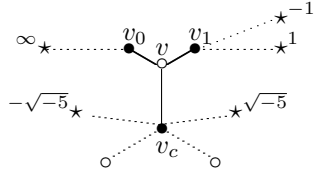
Now Figure \ref{oneballs} shows that precisely four
of the ten vertices containing $\psi(a)$ contain also
$\psi(b)$. Since the white vertex fails to correspond to an
order isomorphic to $\mathbb{M}_2(\mathbb{Z}[\sqrt{-5}])$,
this vertex provides no representations. Each of the remaining
three vertices contributes with two representations, making
a total of six. Two of them can be obtained, as before,
as a conjugate of $\psi$ by a matrix $\mathtt{t}_a=\sbmattrix a210$
for $a\in\{1,\sqrt{-5}\}$, namely:
$$\psi_1(a)=\mathtt{t}_1^{-1}\psi(a)\mathtt{t}_1=
\bmattrix {-1}{-2}11,\qquad
\psi_1(b)=\mathtt{t}_1^{-1}\psi(b)\mathtt{t}_1=
\bmattrix 120{-1},$$
$$\psi_2(a)=\mathtt{t}_{\sqrt{-5}}^{-1}\psi(a)
\mathtt{t}_{\sqrt{-5}}=
\bmattrix {-\sqrt{-5}}{-2}{-2}{\sqrt{-5}},
\quad\textnormal{ and }$$
$$\psi_2(b)=
\mathtt{t}_{\sqrt{-5}}^{-1}\psi(b)\mathtt{t}_{\sqrt{-5}}=
\bmattrix {\sqrt{-5}}23{-\sqrt{-5}},$$
while the remaining three are conjugates of
$\psi$, $\psi_1$ and $\psi_2$ by the matrix 
$\mathtt{u}$ in Example \ref{e67}.
\end{ex}

\begin{ex}
    Consider the quaternion group 
    $H=
    \langle i,j|i^4=e, i^2=j^2, ij=ji^3\rangle$,
    and  the representation 
    $\psi:H\rightarrow\mathrm{GL}_2(K)$,
    where $K=\mathbb{Q}(\sqrt{-1})$, given 
    by the matrices $$\psi(i)=
    \bmattrix 01{-1}0,\qquad \psi(j)=
    \bmattrix {\sqrt{-1}}00{-\sqrt{-1}}.$$
    Since each of these matrices have eigenvalues
    $\pm\sqrt{-1}$, the local branch at any place
    $\nu$ of either matrix is
    a tubular neighborhood of width $\nu(2)$ 
    of a maximal path. 
    The visual limits of each
    maximal path are computed as the fixed
    points of the corresponding Moebius 
    transformation $\mu_1(z)=-1/z$ or
    $\mu_2(z)={-z}$. The first path goes from
    $\sqrt{-1}$ to $-\sqrt{-1}$ and the second 
    from $0$ to $\infty$.
    They intersect precisely at the distinguished vertex
    $v_0$ for every non-dyadic place $\nu$.
    When $\nu$ is the dyadic place, i.e., the
    principal ideal generated by $\pi_\nu=1+\sqrt{-1}$,
    the intersection of the branches is a
    ball of radius $1$ around the vertex $v_1$
    corresponding to the ball 
    $B_{\sqrt{-1}}^{[1]}$, 
    if we normalize the valuation by $\nu(\pi_\nu)=1$,
    as ilustrated in Figure \ref{f11}.
    \begin{figure}
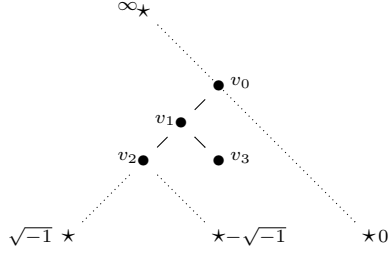

    \[\xygraph{!{<0cm,0cm>;<1cm,0cm>:<0cm,1cm>::}
!{(1,3) }*+{\star}="ei"
!{(0,0) }*+{\star}="e1"
!{(2,0) }*+{\star}="em"
!{(4,0) }*+{\star}="e0"
!{(0.8,3) }*+{{}^{\infty}}
!{(-.5,0) }*+{{}_{\sqrt{-1}}}
!{(2.5,0) }*+{{}_{-\sqrt{-1}}}
!{(4.2,0) }*+{{}_{0}}
!{(1,1) }*+{\bullet}="v2"
!{(1.5,1.5) }*+{\bullet}="v1"
!{(2,2) }*+{\bullet}="v0"
!{(2,1) }*+{\bullet}="v3"
!{(0.8,1) }*+{{}^{v_2}}
!{(1.3,1.5) }*+{{}^{v_1}}
!{(2.3,2) }*+{{}^{v_0}}
!{(2.3,1) }*+{{}^{v_3}}
"ei"-@{.}"e0" "e1"-@{.}"v2" 
"em"-@{.}"v2" "v2"-"v1" "v1"-"v0"
"v1"-"v3"
    }\]
        \caption{The branch of the canonical representation for
        the quaternion group.}\label{f11}
    \end{figure}
        Note that the 
    distinguished vertex $v_0$ is a leaf of this
    branch. We conclude that there are precisely
    $4$ conjugacy classes of integral 
    representations. Other than the class of
    $\psi_0=\psi$, which corresponds to the 
    distinguished vertex $v_0$, the remaining
    three conjugacy classes can be computed
    by finding Moebius transformation sending
    $v_0$ to each of the remaining vertices in 
    Figure \ref{f11}. For example, to send
    $v_0$ to $v_t$, $t=1,2,3$, we can use the 
    Moebius transformation 
    $\alpha_t(z)=a_t(z-b_t)+b_t$, where $b_t$ is 
    a center of the corresponding ball 
    and $|a_t|$ is the radius. In particular we can use $(a_1,b_1)=(\pi_\nu,\sqrt{-1})$,
    $(a_2,b_2)=(2,\sqrt{-1})$ and 
    $(a_3,b_3)=(2,1)$. Note that $1$ is a center
    of the ball corresponding to $v_3$ since
    $\pi_\nu=1+\sqrt{-1}$ is a uniformizer.
    This gives us the representatives $\psi_t$
    defined by 
    $\psi_t(i)=\mathtt{a}_t^{-1}\psi(i)\mathtt{a}_t$
    and $\psi_t(j)=\mathtt{a}_t^{-1}\psi(j)\mathtt{a}_t$,
    where $\mathtt{a}_t=\sbmattrix {a_t}{b_t(1-a_t)}01$.
    The explicit representatives computed in this way are as follows:
    $$\psi_1(i)=\bmattrix 1{\bar\pi_\nu}{-\pi_\nu}{-1}, \qquad
    \psi_1(j)=
    \bmattrix {\sqrt{-1}}{\bar\pi_\nu\sqrt{-1}}
    0{-\sqrt{-1}},$$
      $$\psi_2(i)=
      \bmattrix {-\sqrt{-1}}0{-2}{\sqrt{-1}}, 
      \qquad
    \psi_2(j)=
    \bmattrix {\sqrt{-1}}10{-\sqrt{-1}},$$
      $$\psi_3(i)=\bmattrix {-1}1{-2}1, \qquad
    \psi_3(j)=
    \bmattrix {\sqrt{-1}}{-\sqrt{-1}}0
    {-\sqrt{-1}}.$$
    Alternatively, conjugation by the matrix
    $\Omega=\frac12\psi(-1+i+j+ij)$, whose 
    cube is the identity, permutes cyclically 
    the three images $\psi(i)$, $\psi(j)$ 
    and $\psi(ij)$. 
    This implies that the action of the 
    corresponding Moebius transformation permutes
    the stems of their branches, while leaving
    the branch of the whole group invariant.
    We conclude that $\Omega$, which is a unit 
    over $\mathcal{O}_K$, permutes cyclically
    the three vertices $v_0$, $v_2$ and $v_3$,
    while leaving the central vertex $v_1$
    invariant. In particular $\Omega$ is contained
    in the maximal order corresponding to $v_1$,
    so the representation $\psi_1$ can be 
    extended to a representation of a group of 
    order $24$ that is, in fact, isomorphic to
    $\mathrm{PSL}(2,3)$. Since $\Omega$ is 
    contained in a unique maximal order, locally at the
    dyadic place, only the representations that are conjugate to
    $\psi_1$ have such extension.
\end{ex}

\begin{proof}[Proof of Theorem \ref{t4}]
We consider the group 
$D_n=\langle a,b|a^n=b^2=e, bab=a^{-1}\rangle$,
and we let $K_n$ be the field of definition for all 
faithful two dimensional absolutely irreducible 
representations of $D_n$, which is the real field 
$K=\mathbb{Q}(\xi+\xi^{-1})$, where
$\xi$ is a primitive $n$-th root of 
unity. These are the representation
$\psi_{n,k}$ that can be defined over 
$L=\mathbb{Q}(\xi)$ by the formulas
$$\psi_{n,k}(a)=R'=
\bmattrix {\xi^k}00{\xi^{-k}},
\qquad\psi_{n,k}(b)=S=\bmattrix0110,$$
the former being a conjugate of the representation
$$\tilde{\psi}_{n,k}(a)=R=
\bmattrix 0{-1}1{\xi^k+\xi^{-k}},
\qquad\tilde{\psi}_{n,k}(b)=S=\bmattrix0110,$$
which does have coefficients in $K$.
In fact, conjugation by
$S$ takes each, $R$ and $R'$, to its inverse, 
so both $\tilde{\psi}_{n,k}$ and $\psi_{n,k}$ are
indeed representations of $D_n$. As they have the same
character, they are $\mathrm{GL}_2(L)$-conjugates.
We can assume that $k$ is relatively prime
to $n$, since otherwise we can replace
$D_n$ by a quotient.

The matrix expressing the basis $\{I,R,S,RS\}$
in terms of the canonical basis of the matrix 
algebra is 
$$A=\left(\begin{array}{cccc}1&0&0&-1\\
0&-1&1&0\\0&1&1&\rho\\
1&\rho&0&1\end{array}\right),$$
whose determinant is $\rho^2-4=(\rho-2)(\rho+2)$.
This is a unit unless $n$ is either a prime power
or twice a prime power according to
Lemma \ref{lc4}. In any other case the 
representation generates a maximal order
and therefore it is contained in a unique 
such order.

In all that follows we assume $n$ is either a 
prime power or twice a prime power. The order
generated by the matrices $R$ and $S$ fail to be 
a maximal order at the unique place $\nu$ over
a prime $p$, which is generated by either
$\rho-2$ or $\rho+2$ (or both if $p=2$).
We write $\pi=\rho-2$ if $n$ is a prime power
and $\pi=\rho+2$ if $n$ is twice an odd prime 
power. Then $\nu=\nu_\pi$ is the place corresponding to the 
maximal ideal $P[\nu]=(\pi)$.
Note that the case $n=2$ need not be
included
since the Klein group is abelian and therefore
it has no completely irreducible two dimensional
representation.

To compute the branch at $\nu$ of this representation
we consider the Moebius transformations 
corresponding to the matrices $R$ and $S$.
They are $\iota_S(z)=\frac1z$ 
and $\iota_R(z)=\frac{-1}{z+\rho}$, where 
$\rho=\xi^k+\xi^{-k}$. The first one has
$1$ and $-1$ has invariant points on the 
projective line. The second one has $\xi^k$
and $\xi^{-k}$, which are not defined over 
$K_\nu$, but over a ramified quadratic extension
$L_{\nu'}/K_\nu$. Note that $\pi_{\nu'}=1\pm\xi$ is
a uniformizing parameter for $L_{\nu'}$. 

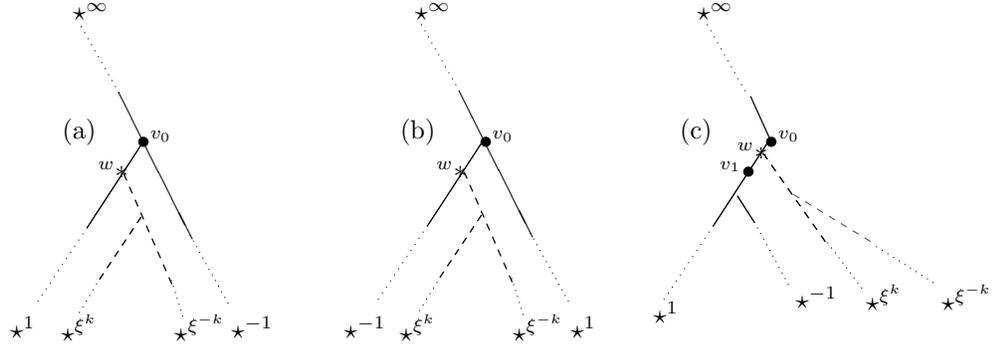
\begin{figure}
\[
\unitlength 1mm 
\linethickness{0.4pt}
\ifx\plotpoint\undefined\newsavebox{\plotpoint}\fi 
\begin{picture}(30,30)(0,12)
\put(5,27){(a)}
\put(17,38){$\star^{\infty}$}
\put(5,14){$\star_1$}
\put(12,12){${}_{\xi^k}$}\put(14,14){$\star$}
\put(21,14){$\star$}\put(20,12){${}_{\xi^{-k}}$}
\put(27,14){$\star_{-1}$}
\put(17,25.5){$\bullet^{v_0}$}
\put(13,24){${}^w*$}
\put(10,22){${}^{v_1}\bullet$}
\multiput(16.43,24.93)(.0334821,-.0513393){16}{\line(0,-1){.0513393}}
\multiput(17.501,23.287)(.0334821,-.0513393){16}{\line(0,-1){.0513393}}
\multiput(18.573,21.644)(.0334821,-.0513393){16}{\line(0,-1){.0513393}}
\multiput(19.644,20.001)(.0334821,-.0513393){16}{\line(0,-1){.0513393}}
\multiput(16.43,24.68)(-.028571,-.164286){5}{\line(0,-1){.164286}}
\multiput(16.144,23.037)(-.028571,-.164286){5}{\line(0,-1){.164286}}
\multiput(15.858,21.394)(-.028571,-.164286){5}{\line(0,-1){.164286}}
\multiput(15.573,19.751)(-.028571,-.164286){5}{\line(0,-1){.164286}}
\multiput(13,22.25)(.03968254,.033730159){126}{\line(1,0){.03968254}}
\multiput(18,26.5)(.035460993,-.033687943){141}{\line(1,0){.035460993}}
\put(17.75,26.25){\line(0,1){7}}
\multiput(17.68,32.68)(.04167,.83333){7}{{\rule{.4pt}{.4pt}}}
\multiput(12.43,22.18)(-.63889,-.61111){10}{{\rule{.4pt}{.4pt}}}
\multiput(15.68,18.93)(-.125,-.875){5}{{\rule{.4pt}{.4pt}}}
\multiput(20.18,18.68)(.35,-.65){6}{{\rule{.4pt}{.4pt}}}
\multiput(22.93,21.43)(.5,-.625){9}{{\rule{.4pt}{.4pt}}}
\end{picture}
\qquad\qquad
\unitlength 1mm 
\linethickness{0.4pt}
\ifx\plotpoint\undefined\newsavebox{\plotpoint}\fi 
\begin{picture}(30,30)(0,12)
\put(5,27){(b)}
\put(17,38){$\star^{\infty}$}
\put(5,14){$\star_{-1}$}
\put(12,12){${}_{\xi^k}$}\put(14,14){$\star$}
\put(21,14){$\star$}\put(20,12){${}_{\xi^{-k}}$}
\put(27,14){$\star_1$}
\put(17,25.5){$\bullet^{v_0}$}
\put(13,24){${}^w*$}
\put(10,22){${}^{v_1}\bullet$}
\multiput(16.43,24.93)(.0334821,-.0513393){16}{\line(0,-1){.0513393}}
\multiput(17.501,23.287)(.0334821,-.0513393){16}{\line(0,-1){.0513393}}
\multiput(18.573,21.644)(.0334821,-.0513393){16}{\line(0,-1){.0513393}}
\multiput(19.644,20.001)(.0334821,-.0513393){16}{\line(0,-1){.0513393}}
\multiput(16.43,24.68)(-.028571,-.164286){5}{\line(0,-1){.164286}}
\multiput(16.144,23.037)(-.028571,-.164286){5}{\line(0,-1){.164286}}
\multiput(15.858,21.394)(-.028571,-.164286){5}{\line(0,-1){.164286}}
\multiput(15.573,19.751)(-.028571,-.164286){5}{\line(0,-1){.164286}}
\multiput(13,22.25)(.03968254,.033730159){126}{\line(1,0){.03968254}}
\multiput(18,26.5)(.035460993,-.033687943){141}{\line(1,0){.035460993}}
\put(17.75,26.25){\line(0,1){7}}
\multiput(17.68,32.68)(.04167,.83333){7}{{\rule{.4pt}{.4pt}}}
\multiput(12.43,22.18)(-.63889,-.61111){10}{{\rule{.4pt}{.4pt}}}
\multiput(15.68,18.93)(-.125,-.875){5}{{\rule{.4pt}{.4pt}}}
\multiput(20.18,18.68)(.35,-.65){6}{{\rule{.4pt}{.4pt}}}
\multiput(22.93,21.43)(.5,-.625){9}{{\rule{.4pt}{.4pt}}}
\end{picture}
\qquad\qquad
\unitlength 1mm 
\linethickness{0.4pt}
\ifx\plotpoint\undefined\newsavebox{\plotpoint}\fi 
\begin{picture}(30,30)(0,12)
\put(5,27){(c)}
\put(17,38){$\star^{\infty}$}
\put(4,14){$\star_1$}
\put(6.5,12){${}_{-1}\star$}
\put(17,12){${}_{\xi^k}$}\put(16.5,14){$\star$}
\put(26,14){$\star$}\put(25,12){${}_{\xi^{-k}}$}
\put(17,28.5){$\bullet^{v_0}$}
\put(17,25.5){$*^w$}
\put(12.5,24){${}^{v_1}\bullet$}
\multiput(12,21.25)(.03968254,.033730159){150}{\line(1,0){.03968254}}
\put(17.75,26.25){\line(0,1){7}}
\multiput(17.68,32.68)(.04167,.83333){7}{{\rule{.4pt}{.4pt}}}
\multiput(11.43,21.18)(-.63889,-.61111){10}{{\rule{.4pt}{.4pt}}}
\multiput(17.93,26.18)(.0343137,-.0326797){17}{\line(1,0){.0343137}}
\multiput(19.096,25.069)(.0343137,-.0326797){17}{\line(1,0){.0343137}}
\multiput(20.263,23.957)(.0343137,-.0326797){17}{\line(1,0){.0343137}}
\multiput(21.43,22.846)(.0343137,-.0326797){17}{\line(1,0){.0343137}}
\multiput(22.596,21.735)(.0343137,-.0326797){17}{\line(1,0){.0343137}}
\multiput(19.68,24.18)(-.033333,-.116667){6}{\line(0,-1){.116667}}
\multiput(19.28,22.78)(-.033333,-.116667){6}{\line(0,-1){.116667}}
\multiput(18.88,21.38)(-.033333,-.116667){6}{\line(0,-1){.116667}}
\multiput(14,23)(-.0326087,-.1521739){23}{\line(0,-1){.1521739}}
\multiput(13,19.18)(-.25,-.82143){7}{{\rule{.4pt}{.4pt}}}
\multiput(18.68,20.18)(-.16667,-.75){7}{{\rule{.4pt}{.4pt}}}
\multiput(23.43,20.43)(.41667,-.66667){7}{{\rule{.4pt}{.4pt}}}
\end{picture}
\]
    \caption{The possible  dispositions of the two paths
    used in the proof of Theorem \ref{t4}.} \label{f33}
\end{figure}

Assume first that $p$ is an odd prime and $n=p^r$.
Note that $\xi^k$ is a primitive root, whence
$\xi^k-1$ is a uniformizer of $L_{\nu'}$ by Lemma \ref{lc3}.
Since $1-\xi^{-k}=\xi^{-k}(\xi^k-1)$ and
$\xi^k-\xi^{-k}=\xi^{-k}(\xi^k-1)(\xi^k+1)$,
where the last factor is a unit, the elements
$1$, $\xi^k$ and $\xi^{-k}$ are actually equidistant. 
Furthermore, $1$ and $-1$ are different in the residue 
field, since $\nu$ is non dyadic. If follows that
$1$, $-1$, $\xi^k$ and $\xi^{-k}$ are located
as in Figure \ref{f33}(a). Note that, since the extension
$L_{\nu'}/K_\nu$ is ramified, the vertex $w$ does not 
correspond to a maximal order defined over $K_\nu$.

Since the representation
$\tilde\psi_{n,k}$ has integral coefficients,
the width of the branch of $R$ is large enough
to include the vertex corresponding to the ball
$B_0^{[0]}=B_1^{[0]}$. It is immediate now that it also
include the vertex corresponding to the ball
$B_1^{[1]}$ which lies at the same distance from
the path from $\xi^k$ to $\xi^{-k}$.
It follows that $R$ and $S$ are contained
in the corresponding maximal orders
$\mathfrak{D}_0$ and $\mathfrak{D}_1$, and 
therefore in $\mathfrak{E}=\mathfrak{D}_0
\cap \mathfrak{D}_1$.
Since the determinant of the matrix $A$, 
in this case,
is a uniformizing parameter, we conclude that $R$
and $S$ generate the
order $\mathfrak{E}$,
and therefore the branch, in this case,
contains precisely two maximal orders.

Essentially the same proof holds when
$n=2p^r$ for an odd prime $p$, since in this
case $\xi^k\equiv-1\ (\mathrm{mod}\ \pi_{\nu'})$
for odd $k$, see Figure \ref{f33}(b).
We assume, therefore, that $n>2$
is a power of $2$. In this case we use the fact
that both $\xi^k\pm1$ and $\xi^{-k}\pm1$ are uniformizers,
to see that the paths are disposed as shown in Figure 
\ref{f33}(c). Note that the residue field has two elements,
since $L_{\nu'}/\mathbb{Q}_2$ is totally ramified,
whence $\xi^k$ and $\xi^{-k}$ are closer from each other than
either is from one. The argument to show that $R$ and $S$ are 
contained in the order $\mathfrak{E}$
is the same as before, but now $R$ and $S$ do not generate 
this order, as the determinant
of $A$ is no longer a uniformizer in this case.
To prove that this representation is not 
contained in additional maximal orders, we
prove that the corresponding residual
representations have a unique eigenspace in 
each case, so that Lemma \ref{l31} applies. 
This is immediate for the
representation $\tilde\psi_{n,k}$ since the
matrix $S$, over the (characteristic $2$)
residue field is conjugate to the matrix
$\sbmattrix 1101$. This show that the vertex
corresponding to $\mathfrak{D}_0$ is a
leaf of the branch. To achieve the same
for $\mathfrak{D}_1$, we need to use 
a representation corresponding to that order.
Write $\pi=\rho-2$, which is a uniformizer in
$K_\nu$.
Note that the moebius transformation
$\mu(z)=\frac{\pi-z}z$ takes $\infty$, $0$
and $1$, respectively, onto $-1$, $\infty$
and $\pi-1$. Therefore, it takes the vertex $v_0$
corresponding to $\mathfrak{D}_0$ onto
the vertex $v_1$ corresponding to 
$\mathfrak{D}_1$. A representation corresponding
to $\mathfrak{D}_1$ is given, therefore,
by the matrices $${\sbmattrix {-1}\pi10}^{-1}
R\sbmattrix {-1}\pi10=\sbmattrix {\rho-1}\pi11
\quad\textnormal{and }\quad
{\sbmattrix {-1}\pi10}^{-1}S
\sbmattrix {-1}\pi10=\sbmattrix {-1}\pi01.$$
The image of the first matrix in the residual algebra is $\sbmattrix 1011$, so the result follows as before. This concludes the proof.
\end{proof}

\end{document}